\documentclass[12pt,reqno]{amsart}
\usepackage[activeacute,english]{babel}

\usepackage[utf8]{inputenc}
\usepackage{amsmath} 
\usepackage{amsthm} 
\usepackage{amssymb} 
\usepackage{amscd} 

\usepackage{emptypage}
\usepackage{enumitem} 
\usepackage{mathtools} 
\usepackage{mathrsfs} 
\usepackage{upgreek}

\usepackage{hyperref,cleveref,autonum} 

\usepackage{esint} 
\usepackage{graphicx,caption} 
\usepackage{float} 

\usepackage{pgf,tikz,pgfplots}
\usetikzlibrary{arrows}
\usetikzlibrary{intersections}
\usetikzlibrary{fadings}
\usepackage{wrapfig}
\usepackage{xcolor}
\definecolor{CBgreen}{HTML}{009E73}
\definecolor{CBorange}{HTML}{E69F00}

\usepackage[margin=0.9in]{geometry}
\newtheorem{theorem}{Theorem}[section]
\newtheorem*{theorem*}{Theorem}
\newtheorem{proposition}[theorem]{Proposition}
\newtheorem{lemma}[theorem]{Lemma}
\newtheorem{corollary}[theorem]{Corollary}
\newtheorem{conjecture}[theorem]{Conjecture}

\theoremstyle{definition}

\theoremstyle{remark}
\newtheorem{remark}[theorem]{Remark}
\newtheorem*{remark*}{Remark}

\newcommand{\con}[1]{\mathbb{#1}}
\newcommand{\C}{\con{C}} 
\newcommand{\R}{\con{R}} 
\newcommand{\Z}{\con{Z}} 

\newcommand{\D}{\mathcal{D}}
\newcommand{\RR}{\mathcal{R}}
\newcommand{\re}{\operatorname{Re}}

\newcommand{\Dom}{\mathrm{Dom}}
\newcommand{\id}{\mathrm{Id}}

\numberwithin{equation}{section}
\title[Quantum dot Dirac operators and $\overline\partial$-Robin Laplacians in shape optimization]{A connection between quantum dot Dirac operators\\and $\overline\partial$-Robin Laplacians in the context of\\shape optimization problems}

\author[J. Duran]{Joaquim Duran}
\address{J. Duran
\newline
Centre de Recerca Matem\`atica, Edifici C, Campus Bellaterra, 08193 Bellaterra, Spain}
\email{jduran@crm.cat}

\author[A. Mas]{Albert Mas}
\address{A. Mas \textsuperscript{1,2}
\newline
\textsuperscript{1}
Departament de Matem\`atiques,
Universitat Polit\`ecnica de Catalunya,
Campus Diagonal Bes\`os, Edifici A (EEBE), Av. Eduard Maristany 16, 08019
Barcelona, Spain
\newline
\textsuperscript{2}
Centre de Recerca Matem\`atica, Edifici C, Campus Bellaterra, 08193 Bellaterra, Spain}
\email{albert.mas.blesa@upc.edu}

\author[T. Sanz-Perela]{Tom\'as Sanz-Perela}
\address{T. Sanz-Perela \textsuperscript{1,2}
\newline
\textsuperscript{1}
Departament de Matem\`atiques i Inform\`atica,
Universitat de Barcelona,
Gran Via de les Corts Catalanes 585, 08007, Barcelona, Spain
\newline
\textsuperscript{2}
Centre de Recerca Matem\`atica, Edifici C, Campus Bellaterra, 08193 Bellaterra, Spain}
\email{tomas.sanz.perela@ub.edu}

\date{\today}

\thanks{The three authors are supported by the Spanish grants PID2021-123903NB-I00 and RED2022-134784-T funded by MCIN/AEI/10.13039/501100011033 and by ERDF ``A way of making Europe", and by the Catalan grant 2021-SGR-00087. This work is supported by the Spanish State Research Agency, through the Severo Ochoa and Mar\'ia de Maeztu Program for Centers and Units of Excellence in R\&D (CEX2020-001084-M). The first author is also supported by CEX2020-001084-M-20-1 and acknowledges CERCA Programme/Generalitat de Catalunya for institutional support.
The third author is also supported by the Spanish grant PID2024-156429NB-I00.}

\begin{document}

\begin{abstract}
This work addresses Faber-Krahn-type inequalities for  quantum dot Dirac operators with nonnegative mass on bounded domains in $\R^2$. 
We show that this family of inequalities is equivalent to a family of Faber-Krahn-type inequalities for $\overline\partial$-Robin Laplacians. Thanks to this, we prove them in the case of simply connected domains for quantum dot boundary conditions asymptotically close to zigzag boundary conditions. Finally, we also study the case of negative mass.

\end{abstract}

\maketitle 

\setcounter{tocdepth}{2}
\makeatletter
\def\l@subsection{\@tocline{2}{0pt}{2.5pc}{5pc}{}}
\makeatother
\tableofcontents

\section{Introduction}

Dirac operators with different kind of boundary conditions ---such as infinite mass, zigzag, or armchair boundary conditions--- are commonly used in the physics literature to model both the confinement of (quasi-)particles in planar regions and electrons conducting electricity in graphene quantum dots and nano-ribbons \cite{AkhmerovBeenakker,mccannfalko04,WurmRycerzetAl}.
Two-dimensional Dirac operators with these type of boundary conditions have been studied from a mathematical perspective; see for example \cite{Antunes2021,Benguria2017Self,Benguria2017Spectral,Schmidt1995}. 
Similarly, three-dimensional Dirac operators with analogous boundary conditions have also been investigated; see, for instance, \cite{Mas2022,DuranMas2024,Holzmann2021,Vega2016}.

In this context, Faber-Krahn-type inequalities are conjectured for the first nonnegative eigenvalue, both in dimension two \cite[Problem 5.1]{ProblemListShapeOptimization} and dimension three~\cite[Conjecture~1.8]{Mas2022}. 
In the present work, we consider the two-dimensional setting, and we show that such Faber-Krahn-type inequalities for quantum dot Dirac operators with nonnegative mass are equivalent to corresponding inequalities for the $\overline\partial$-Robin Laplacians studied in \cite{Duran2025}. 
As a consequence, we prove that among bounded simply connected domains with the same area, disks are the unique minimizers (in an asymptotic sense made precise below) of the first nonnegative eigenvalue of quantum dot Dirac operators with boundary conditions sufficiently close to the zigzag case. Our approach to prove this result also allows us to address a shape optimization problem for quantum dot Dirac operators with negative mass.

\subsection{Quantum dot Dirac operators}

To set the stage, throughout the present work $\Omega\subset \R^2$ will be a bounded domain with $C^2$~boundary. 
We denote by $\nu=(\nu_1,\nu_2)$ the unit normal vector field at $\partial\Omega$ which points outwards of $\Omega$. 
We set $\tau=(\tau_1,\tau_2) := (-\nu_2, \nu_1)$; in this way, $\tau$ is the unit vector field tangent to $\partial\Omega$ such that $\{\nu,\tau\}$ is positively oriented. 
Based on the identification $\R^2\equiv\C$, in the sequel it will be convenient to abuse notation as follows:
$$\R^2\ni\nu =(\nu_1,\nu_2)\equiv \nu_1 + i\nu_2=\nu\in\C$$
and, accordingly, we will write $\overline \nu =\nu_1 - i\nu_2$; to which notation we are referring to will always be clear from the context. Moreover, we will use the complex notation
$\partial_z := \frac{1}{2} (\partial_1 -i \partial_2)$ and $\partial_{\bar z} := \frac{1}{2} (\partial_1 +i \partial_2)$, where $\nabla:=(\partial_1,\partial_2)$ denotes the gradient in $\R^2$.

Let $-i \sigma \cdot \nabla + m \sigma_3$ denote the differential expression of the free Dirac operator in $\R^2$. Here, $m\in\R$ typically denotes the mass, $\sigma := (\sigma_1, \sigma_2)$, and
$$
\sigma_1 := 
\begin{pmatrix}
	0&1\\1&0
\end{pmatrix},
\quad
\sigma_2 := 
\begin{pmatrix}
	0&-i\\i&0
\end{pmatrix},
\quad
\sigma_3 := 
\begin{pmatrix}
	1&0\\0&-1
\end{pmatrix}
$$
are the Pauli matrices ---as customary, we denote 
$\sigma \cdot p:=\sigma_1 p_1+\sigma_2 p_2$ for $p = (p_1,p_2)$.

Motivated by their applications in the description of graphene quantum dots and nano-ribbons, the following family of {\bf quantum dot Dirac operators} was studied in \cite{Benguria2017Self,Benguria2017Spectral} (for the case $m=0$).
Given $\theta\in(-\frac \pi 2,\frac {3\pi}{2})\setminus\{\frac {\pi}{2}\}$, let $\D_\theta$ be the operator in $L^2(\Omega)^2$ defined\footnote{In \Cref{sec:Notation} we recall some basic  definitions and standard notation to be used throughout the paper.} by
\begin{equation}\label{def:Dirac_op_theta}
	\begin{split}
		\mathrm{Dom}(\D_\theta) &:= \big\{ \varphi \in H^1(\Omega)^2: \, \varphi =(\cos\theta\,\sigma\cdot \tau+\sin\theta\, \sigma_3) \varphi  \,\text{ in } H^{1/2}(\partial \Omega)^2 \big\},\\
		\D_\theta\varphi &:= 
		(-i\sigma\cdot\nabla+m\sigma_3)\varphi \quad\text{for all 
			$\varphi\in\mathrm{Dom}(\D_\theta)$.}
	\end{split}
\end{equation}
The {\bf infinite mass} boundary conditions correspond to $\theta\in\{0, \pi\}$. 
The values $\theta\in\{-\frac \pi 2, \frac\pi 2\}$ excluded in the previous definition give rise to the so-called  {\bf zigzag} boundary conditions, which in \eqref{def:Dirac_op_theta} formally lead to either
$\varphi =- \sigma_3 \varphi$ or $\varphi = \sigma_3 \varphi$ on $\partial\Omega$ (forcing one of the components of~$\varphi$ to vanish at $\partial \Omega$).
The self-adjoint realization of these operators is
\begin{equation}
\begin{split}
\mathrm{Dom}(\D_{-\frac \pi 2}) 
&:= \big\{ \varphi= (u, v)^\intercal : \, u \in H^1_0(\Omega),\,  v \in L^2(\Omega),\, \sigma\cdot\nabla \varphi \in L^2(\Omega)^2 \big\},\\
\D_{-\frac \pi 2}\varphi &:= (-i\sigma\cdot\nabla + m \sigma_3)\varphi \quad\text{for all $\varphi\in\mathrm{Dom}(\D_{-\frac \pi 2})$},
\end{split}
\end{equation}
and
\begin{equation}
\begin{split}
\mathrm{Dom}(\D_{\frac \pi 2}) 
&:= \big\{ \varphi= (u, v)^\intercal : \, u \in L^2(\Omega),\, v \in H^1_0(\Omega),\, \sigma\cdot\nabla\varphi \in L^2(\Omega)^2 \big\},\\
\D_{\frac \pi 2}\varphi &:= (-i\sigma\cdot\nabla + m \sigma_3)\varphi \quad\text{for all $\varphi\in\mathrm{Dom}(\D_{\frac \pi 2}).$}
\end{split}
\end{equation} 
Note that in these cases $\varphi$ need not be in $H^1(\Omega)^2$.

For the case in which $m=0$ and $\theta\in(-\frac \pi 2,\frac {3\pi}{2})\setminus \{\frac \pi 2\}$, in \cite{Benguria2017Self} it is proven that $\D_\theta$ is self-adjoint in $L^2(\Omega)^2$, and that its spectrum\footnote{In the sequel, we will denote by $\sigma(T)$ the spectrum of a given operator $T$. This notation should not be confused with the vector of Pauli matrices $\sigma = (\sigma_1, \sigma_2)$.} consist of eigenvalues of finite multiplicity accumulating only at $\pm \infty$. Still in the case $m=0$, in \cite{Schmidt1995} it is proven that $\D_{-\frac \pi 2}$ is self-adjoint in $L^2(\Omega)^2$, and that $0$ is an eigenvalue of infinite multiplicity; by unitary equivalence, the same holds true for $\D_{\frac \pi 2}$ with $m=0$. Analogous conclusions hold for the general case when $m\in \R$: since the Pauli matrix $\sigma_3$ is self-adjoint, by the Kato-Rellich theorem \cite[Theorem 4.3 in Section 5.4.1]{Kato1995} the operator $\D_\theta$ is self-adjoint in $L^2(\Omega)^2$ for every $\theta\in[-\frac \pi 2,\frac {3\pi}{2})$. It has purely discrete spectrum when $\theta\neq \pm \frac \pi 2$, and $\pm m$ is an eigenvalue of $\D_{\pm \frac \pi 2}$ of infinite multiplicity. Furthermore, as a consequence of \cite[Proposition 3]{Schmidt1995}, it holds that 
\begin{equation} \label{eq:SpectrumZigZag}
    \sigma(\D_{\pm \frac \pi 2}) = \{\pm m\} \cup \big\{ \sqrt{\Lambda+m^2},-\sqrt{\Lambda+m^2}: \Lambda\in \sigma(-\Delta_{\mathrm{D}}) \big\},
\end{equation}
where $-\Delta_{\mathrm{D}}$ denotes the self-adjoint realization of the (positive) Dirichlet Laplacian in $L^2(\Omega)$. 

Before proceeding further, let us make some considerations based on unitary equivalences to delimit a bit the spectral study of $\D_\theta$. First, since the boundary condition in \eqref{def:Dirac_op_theta} is $2\pi$-periodic in $\theta$, it is enough to consider $\theta\in[-\frac \pi 2,\frac {3\pi}{2})$ to cover the whole range of parameters $\theta\in\R$.
Furthermore, even if due to physical considerations one usually  assumes $m\geq0$, by the chiral transformation described in \Cref{sec:Invariance} it is enough to study the operator $\D_\theta$ for $\theta\in[-\frac \pi 2,\frac {\pi}{2}]$, but now considering both the cases $m\geq 0$ and $m<0$ (as we will see, these two cases give rise to qualitatively different situations).
Finally, by charge conjugation (see again \Cref{sec:Invariance}) it suffices to study the nonnegative part of the spectrum of $\D_\theta$.

Taking into account these considerations, for $m\in\R$ and $\theta\in[-\frac \pi 2,\frac {\pi}{2}]$, we shall denote the first (smallest) nonnegative eigenvalue of $\D_\theta$ by $\lambda_\Omega(\theta)$, that is,
\begin{equation}\label{def:1stDiracEigenvalue}
	\lambda_\Omega(\theta):=\min\big(\sigma(\D_\theta)\cap[0,+\infty)\big).
\end{equation} 
Note that we use the subscript to highlight the dependence on $\Omega$.
Likewise, in the sequel we will denote the first (smallest) eigenvalue of the Dirichlet Laplacian $-\Delta_{\mathrm{D}}$ by $\Lambda_\Omega$, that is, 
\begin{equation} \label{def:1stDirichletEigenvalue}
    \Lambda_\Omega := \min \sigma(-\Delta_{\mathrm{D}}).
\end{equation}

\subsection{The shape optimization problem}

A usual problem in spectral geometry is to optimize certain spectral quantities (such as eigenvalues) among all bounded domains satisfying a geometric constraint.
In the context of {\em generalized MIT bag models} in $\R^3$, in \cite[Conjecture 1.8]{Mas2022} it is conjectured that, among all bounded $C^2$ domains with prescribed volume, the first nonnegative eigenvalue of the underlying operator is minimal for a ball. 
In view of the correspondence described in \cite[Section 1.3]{DuranMas2024} between three-dimensional and two-dimensional settings, in the context of quantum dot Dirac operators in $\R^2$ this conjecture reads as follows.

\begin{conjecture} \label{Conj:Dirac}
    Assume that $m\geq0$. Let $\Omega\subset\R^2$ be a bounded domain with $C^2$ boundary and let $D\subset\R^2$ be a disk with the same area as 
    $\Omega$. If $\Omega$ is not a disk, then
    $\lambda_\Omega(\theta)> \lambda_{D}(\theta)$ for all $\theta\in(-\frac \pi 2,\frac \pi 2)$.
\end{conjecture}

The important case $m=\theta=0$ in \Cref{Conj:Dirac} is known as \emph{the Faber-Krahn inequality for the Dirac operator with infinite mass boundary condition}. It is considered a hot open problem in spectral geometry; see \cite[Problem 5.1]{ProblemListShapeOptimization}. A first geometrical lower bound, which is never attained among Euclidean domains, was established in \cite{Benguria2017Spectral}. Partial optimization results for rectangles were shown in \cite{BrietK2022}.
An important contribution towards the possible resolution of the conjecture can be found in \cite[Theorem 4]{Antunes2021}, in which a variational characterization of $\lambda_\Omega(0)$ is given ---see \Cref{rmk:AntunesEtAl} for a comparison with our results.
Using this characterization, in \cite[Section~8]{Antunes2021} some numeric simulations are done and the results strongly support the validity of \Cref{Conj:Dirac}, at least for case $m=\theta=0$. Further numerical analysis for this and related shape optimization problems can be found in \cite{Antunes2024}. 

\begin{remark}\label{rmk:RangeConjectureQD}
	As the reader may have noticed, \Cref{Conj:Dirac} is only posed for $\theta\in(-\frac \pi 2,\frac \pi 2)$ and for $m\geq 0$.
	This a priori comes from the range of parameters studied in \cite{Mas2022} taking into account the correspondence of \cite[Section 1.3]{DuranMas2024} mentioned before, but let us comment on why $\theta\in(-\frac \pi 2,\frac \pi 2)$ and $m\geq 0$ is the natural range of parameters in which \Cref{Conj:Dirac} makes sense.
	
	First, note that for $m\geq0$ the zigzag cases $\theta=\pm \frac \pi 2$ are excluded in the statement of \Cref{Conj:Dirac} for obvious reasons. 
	On the one hand, if $\theta= \frac \pi 2$, by \eqref{eq:SpectrumZigZag} we have $\lambda_\Omega(\frac\pi 2)=m$ independently of the shape of $\Omega\subset\R^2$ ---the same conclusion holds true if $\theta= -\frac \pi 2$ and $m=0$. 
	On the other hand, if $\theta= -\frac \pi 2$ and $m>0$, by \eqref{eq:SpectrumZigZag} we have $\lambda_\Omega(-\frac\pi 2)= \sqrt{\Lambda_\Omega+m^2}$, which is minimal for a disk among all bounded $C^2$ domains with prescribed area, by the Faber-Krahn inequality \cite{Faber1923,Krahn1925}.
	
	The reason to exclude the range $(\frac \pi 2,\frac {3\pi}{2})$ for $m\geq 0$ (or, equivalently, to exclude the case $m<0$ for $\theta\in(-\frac \pi 2,\frac \pi 2)$; see \Cref{sec:Invariance} and also compare \Cref{Fig:EVCurvesQDDiracFullRange} with \Cref{Fig:EVCurvesQDDirac(NegativeMass)}) is the following:
	while for $\theta\in(-\frac \pi 2,\frac \pi 2)$ and $m\geq 0$ the eigenvalues are always outside the interval $[-m, m]$ (see\footnote{\Cref{lm:lambda>m} states that any nonnegative eigenvalue must be strictly bigger than $m$. Then, the result for negative eigenvalues follows from charge conjugation; see \Cref{sec:Invariance}.} \Cref{lm:lambda>m}), this is no longer true for $\theta \in (\frac \pi 2,\frac {3\pi}{2})$ and $m\geq 0$.
	In fact, for $\theta \in (\frac \pi 2,\frac {3\pi}{2})$ it may make no sense to optimize $\lambda_\Omega(\theta)$.

	To visualize this better, let us show a description of the spectrum of $\D_{\theta}$ when $\Omega$ is a disk.
	In this case, one can explicitly compute the eigenfunctions of $\D_{\theta}$ and obtain implicit equations for the eigenvalues,\footnote{This can be done following \cite[Appendix~A]{LotoreichikOurmieresBonafos}, or also using \cite[Appendix~A]{Duran2025} taking into account the connection between quantum dot Dirac operators and $\overline\partial$-Robin Laplacians described in \Cref{ss:CbqdDoRL}.} which can be represented graphically as functions of $\theta$, as done in \Cref{Fig:EVCurvesQDDiracFullRange}.

	\begin{figure}[h]
		\includegraphics[width=0.85\textwidth]{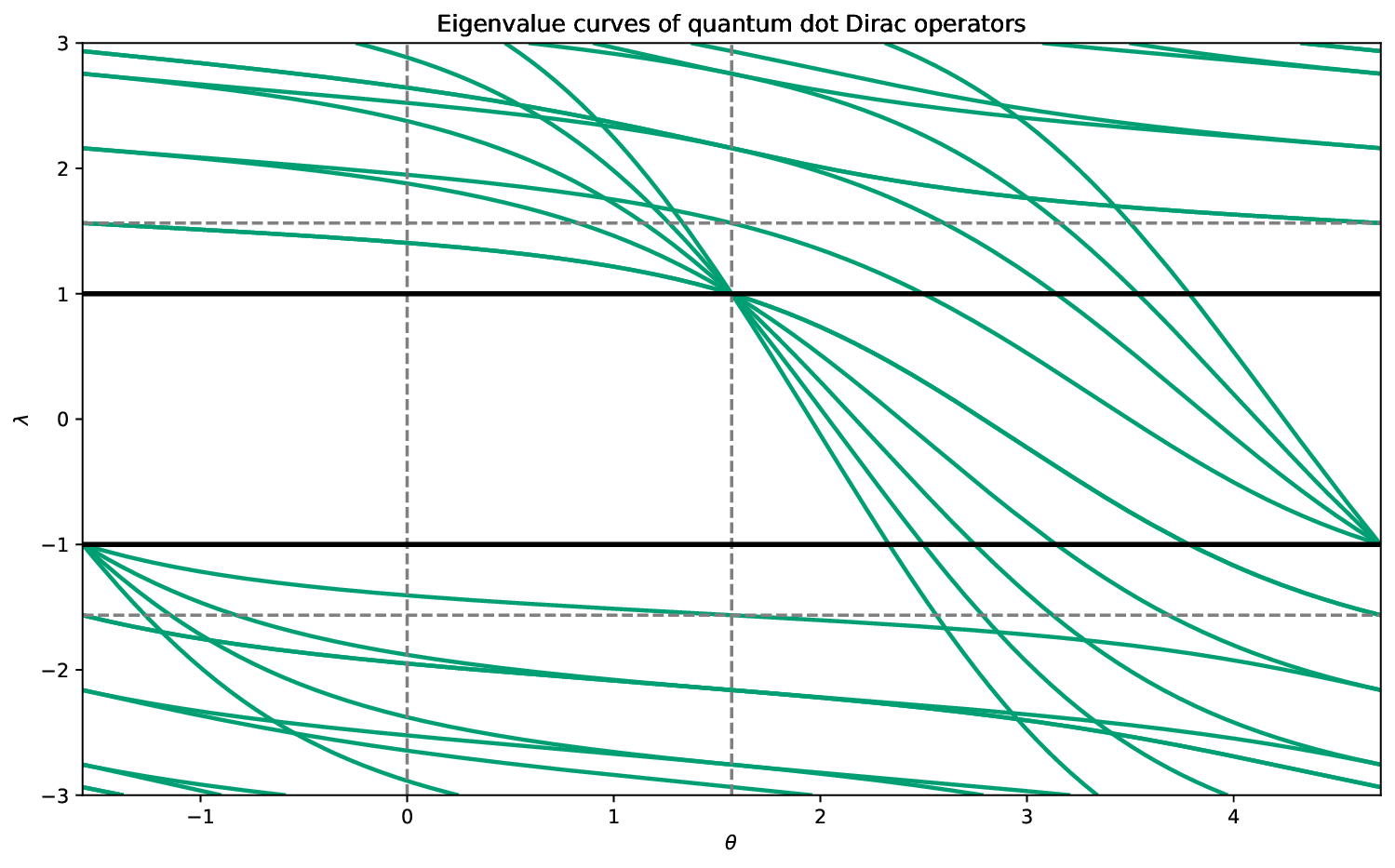}
		\caption{Eigenvalue curves of the quantum dot Dirac operators $\D_{\theta}$ for the disk $D_R$ of radius $R=2$ and with $m=1$.
		We have included two solid horizontal black lines to highlight the location of $\pm m$.
			The dashed vertical lines at $\theta=0$ and $\theta=\pi/2$ help to locate, respectively, the infinite mass and zigzag cases.
			The horizontal dashed lines represent the values $\pm \sqrt{\Lambda_{D_R}+m^2}$.  }
		\label{Fig:EVCurvesQDDiracFullRange}
	\end{figure}

	We should mention that in the previous picture we only plotted some eigenvalues, but actually there are infinitely many eigenvalue curves converging to $m$ as $\theta \uparrow \frac \pi 2$, accounting for the infinite multiplicity of $m$ as eigenvalue of $\D_{\frac \pi 2}$.
	Each of these curves ---always above the plotted ones in the range $(-\frac \pi 2,\frac {\pi}{2})$--- can be continued smoothly in $(\frac \pi 2,\frac {3\pi}{2})$, as the ones in \Cref{Fig:EVCurvesQDDiracFullRange}.
	As a consequence, if we extend the definition of $\lambda_\Omega(\theta)$ in \eqref{def:1stDiracEigenvalue} to $(-\frac \pi 2,\frac {3\pi}{2})$, the function $\theta \mapsto \lambda_\Omega(\theta)$, which in $(-\frac \pi 2,\frac {\pi}{2})$ is continuous and decreasing, in $(\frac \pi 2,\frac {3\pi}{2})$ has infinitely many jump discontinuities accumulating at $\frac \pi 2$. This indicates that $\lambda_\Omega(\theta)$ is not the right quantity to optimize in $(\frac \pi 2,\frac {3\pi}{2})$ for $m\geq 0$.
	
	Let us advance that in \Cref{subsec:neg_mass} we will study the optimization of another spectral quantity in $(\frac \pi 2,\frac {3\pi}{2})$ for $m\geq 0$, which roughly speaking will be the crossing point with the level set $-m$ of the natural continuation of $\theta \mapsto \lambda_\Omega(\theta)$ from $(-\frac \pi 2,\frac {\pi}{2})$ to $(\frac \pi 2,\frac {3\pi}{2})$.
\end{remark}

Our main motivation for this work is to address \Cref{Conj:Dirac} through a connection with another family of operators.
As we will see in more detail in \Cref{sec:TheProblem}, for $m\geq0$, the nonnegative eigenvalues of the family $\{\D_\theta\}_{\theta\in(-\frac \pi 2,\frac {\pi}{2})}$ are related to the eigenvalues of the family $\{\RR_a\}_{a>0}$, where $\RR_a$ is the operator in $L^2(\Omega)$ defined by
    \begin{equation} \label{eq:RodzinLaplacian}
        \begin{split}
            \Dom(\RR_a) & := \big\{u\in H^1(\Omega): \, \partial_{\bar z} u \in H^1(\Omega), \, 2\bar \nu \partial_{\bar z}u + au = 0 \text{ in } H^{1/2}(\partial \Omega) \big\}, \\
            \RR_a u & := -\Delta u \quad \text{for all } u \in \Dom(\RR_a).
        \end{split}
    \end{equation}
This operator is studied in the recent paper \cite{Duran2025}, where it is called the {\bf $\overline\partial$-Robin Laplacian} because of its similarity with the standard Robin Laplacian.\footnote{The sesquilinear form associated to the Robin Laplacian is based on the decomposition $\Delta=\operatorname{div}\nabla$ and integration by parts. Instead, the one for the $\overline\partial$-Robin Laplacian comes from the decomposition $\Delta=4\partial_z\partial_{\bar z}$; see \Cref{rmk:RobinVsRodzin} for more details.} In particular, in \cite[Theorem~1.1]{Duran2025} it is shown that $\RR_a$ is self-adjoint in $L^2(\Omega)$ for every $a>0$, and that its spectrum is purely discrete and strictly positive \cite[Theorem~1.2]{Duran2025}. In \Cref{Fig:EVCurvesRodzin} we graphically represent the eigenvalues of $\RR_a$ in terms of the boundary parameter $a$.
As before, to highlight the dependence on $\Omega$, for $a>0$ we shall denote the first (smallest) eigenvalue of $\RR_a$ by $\mu_\Omega(a)$, that is,
\begin{equation}\label{def:muOmega}
\mu_\Omega(a):=\min\big(\sigma(\RR_a)\big).
\end{equation} 

\begin{figure}
	\includegraphics[width=0.85\textwidth]{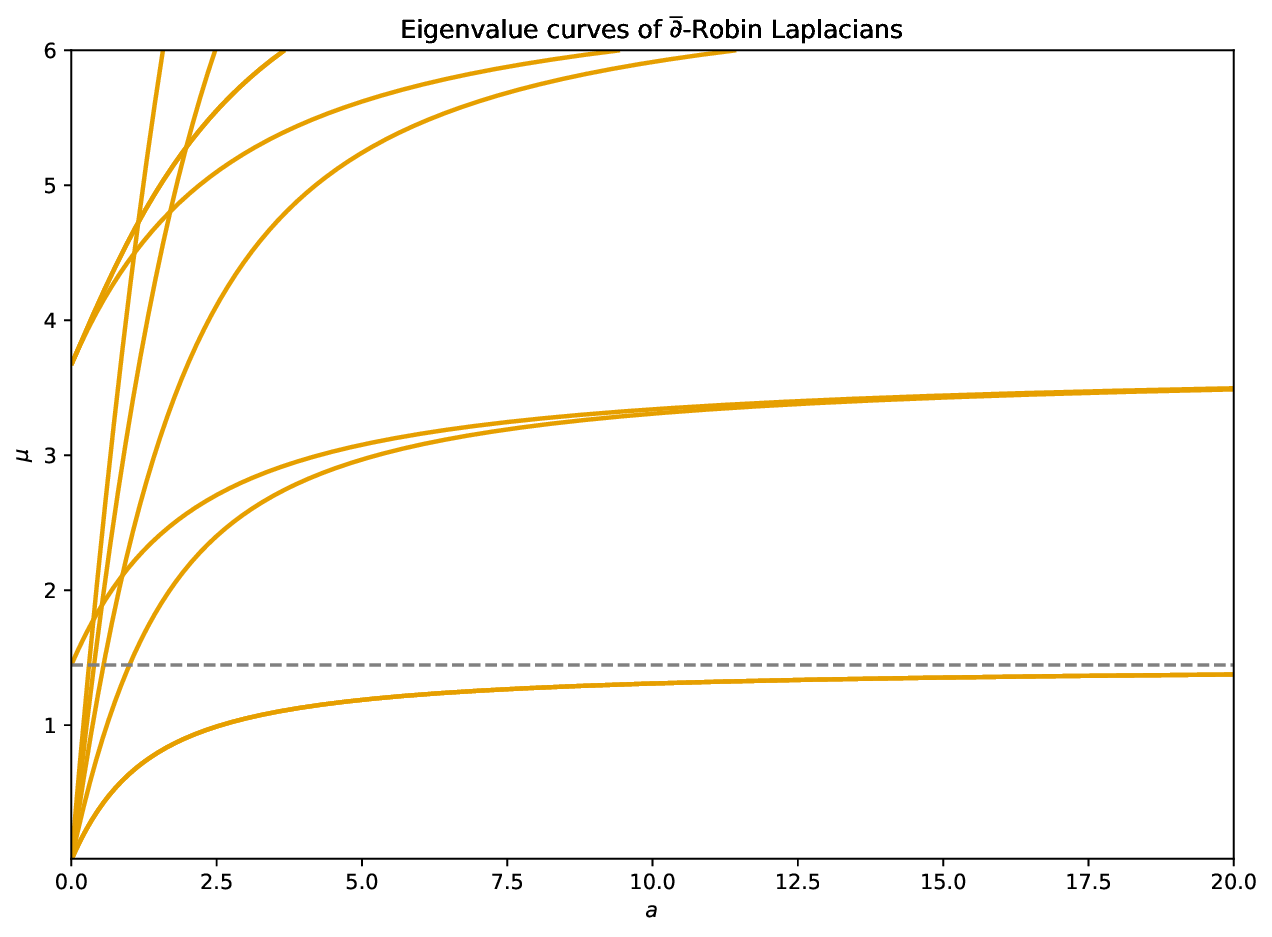}
	\caption{Eigenvalue curves of the $\overline\partial$-Robin Laplacians $\RR_a$ for the disk $D_R$ of radius $R=2$.
	The horizontal dashed line represents the value $\Lambda_{D_R}$.  }
	\label{Fig:EVCurvesRodzin}
\end{figure}

The relation between the nonnegative spectrum of $\{\D_\theta\}_{\theta\in(-\frac \pi 2,\frac {\pi}{2})}$ and the spectrum of $\{\RR_a\}_{a>0}$ mentioned above ---and, in particular, between $\lambda_\Omega$ and $\mu_\Omega$---, naturally leads to the following open problem based on  \Cref{Conj:Dirac}.

\begin{conjecture} \label{Conj:RobinType}
$($\cite[Conjecture 1.10]{Duran2025}$)$
Let $\Omega\subset\R^2$ be a bounded domain with $C^2$ boundary and let $D\subset\R^2$ be a disk with the same area as 
    $\Omega$. If $\Omega$ is not a disk, then
    $\mu_\Omega(a)> \mu_D(a)$ for all $a>0$.
\end{conjecture}

The main purpose of the present work is to show that \Cref{Conj:Dirac} and \Cref{Conj:RobinType} are equivalent, and to prove that they hold true in the asymptotic regimes of the parameters $\theta$ and~$a$.
The precise statements of these results are presented in the next section.

\section{Main results} \label{sec:MainResults}

\subsection{Equivalence of conjectures} \label{sec:Equivalence}

Our first main result is \Cref{Th:eqiv_conj_point} below.
It asserts that \Cref{Conj:Dirac} for quantum dot Dirac operators holds true for a given $\Omega$ and a given $\theta\in(-\frac \pi 2,\frac {\pi}{2})$, provided that \Cref{Conj:RobinType} for $\overline\partial$-Robin Laplacians holds true for the same domain $\Omega$ and for some specific choice of $a>0$ depending on $\Omega$ and $\theta$.
Conversely, it also asserts that \Cref{Conj:RobinType} holds true for a given $\Omega$ and a given $a>0$, if \Cref{Conj:Dirac} holds true for the same domain $\Omega$ and for some specific choice of $\theta\in(-\frac \pi 2,\frac {\pi}{2})$ depending on $\Omega$ and $a$.

To properly state the theorem, we shall need the following qualitative properties of $\mu_\Omega$, the first eigenvalue of $\RR_a$ seen as a function of $a$; they are proven throughout \cite{Duran2025} and the precise statements will be given in \Cref{thm:PropertiesMu}~$(ii)$.  
The function 
$a\mapsto\mu_\Omega(a)$ is continuous, strictly increasing, and bijective from $(0,+\infty)$ to $(0, \Lambda_\Omega)$. In particular, the inverse function $\mu_\Omega^{-1}$ is well-defined, continuous, strictly increasing, and bijective from $(0, \Lambda_\Omega)$ to $(0,+\infty)$. From these properties, we will deduce in \Cref{thm:PropertiesLambda} that the function 
$\theta\mapsto\lambda_\Omega(\theta)$ is continuous, strictly decreasing, and bijective from $(-\frac \pi 2,\frac {\pi}{2})$ to $(m,\sqrt{\Lambda_\Omega+m^2})$. In particular, the inverse function $\lambda_\Omega^{-1}$ is well-defined, continuous,  strictly decreasing, and bijective from $(m,\sqrt{\Lambda_\Omega+m^2})$ to $(-\frac \pi 2,\frac {\pi}{2})$. One can visualize these properties in \Cref{Fig:FirstEigenvalueCurve}.

\begin{figure}[htbp]
	\centering
	\begin{minipage}[b]{0.45\textwidth}
		\centering
		\includegraphics[width=\textwidth]{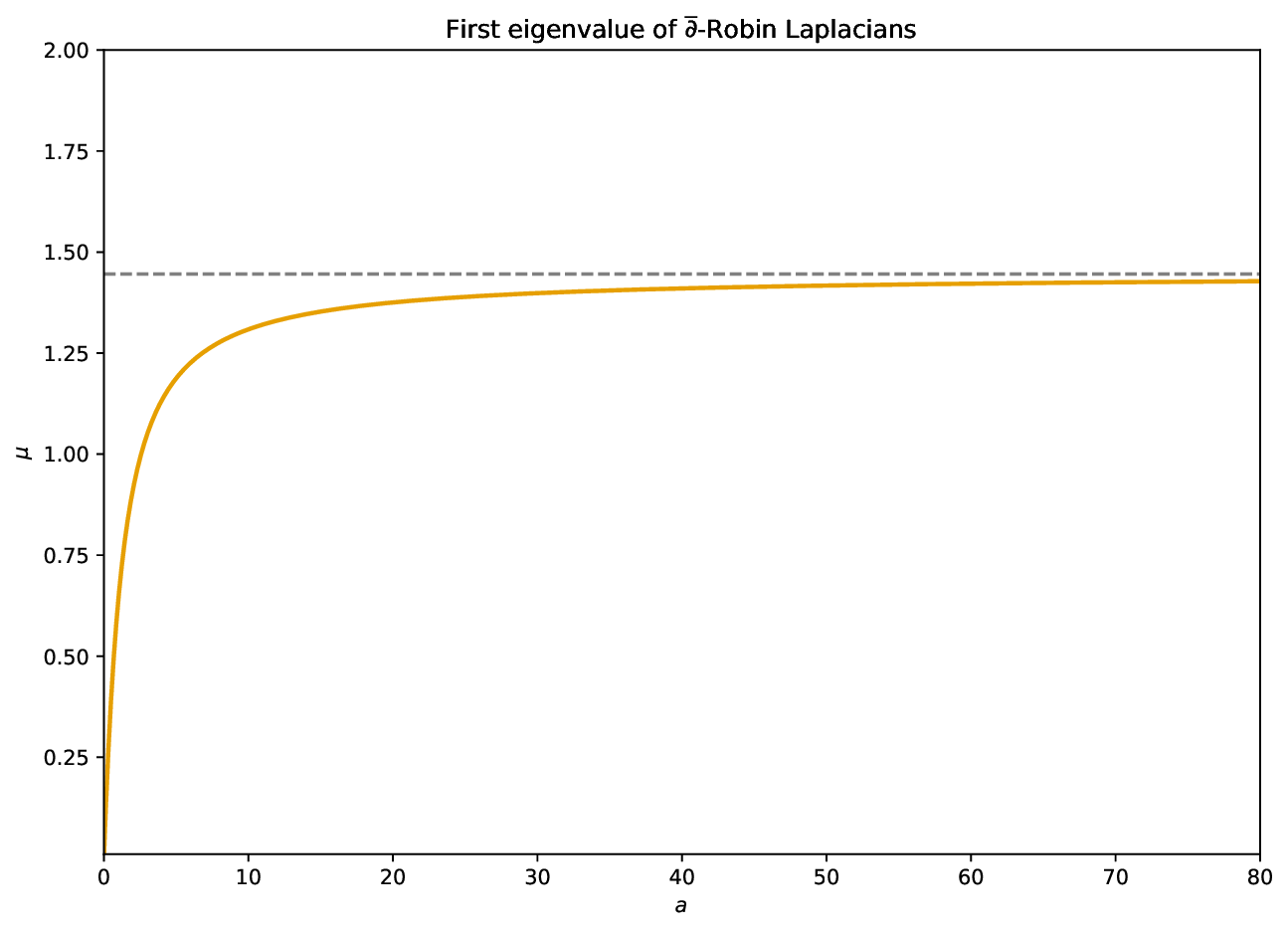}
	\end{minipage}
	\hfill
	\begin{minipage}[b]{0.52\textwidth}
		\centering
		\includegraphics[width=\textwidth]{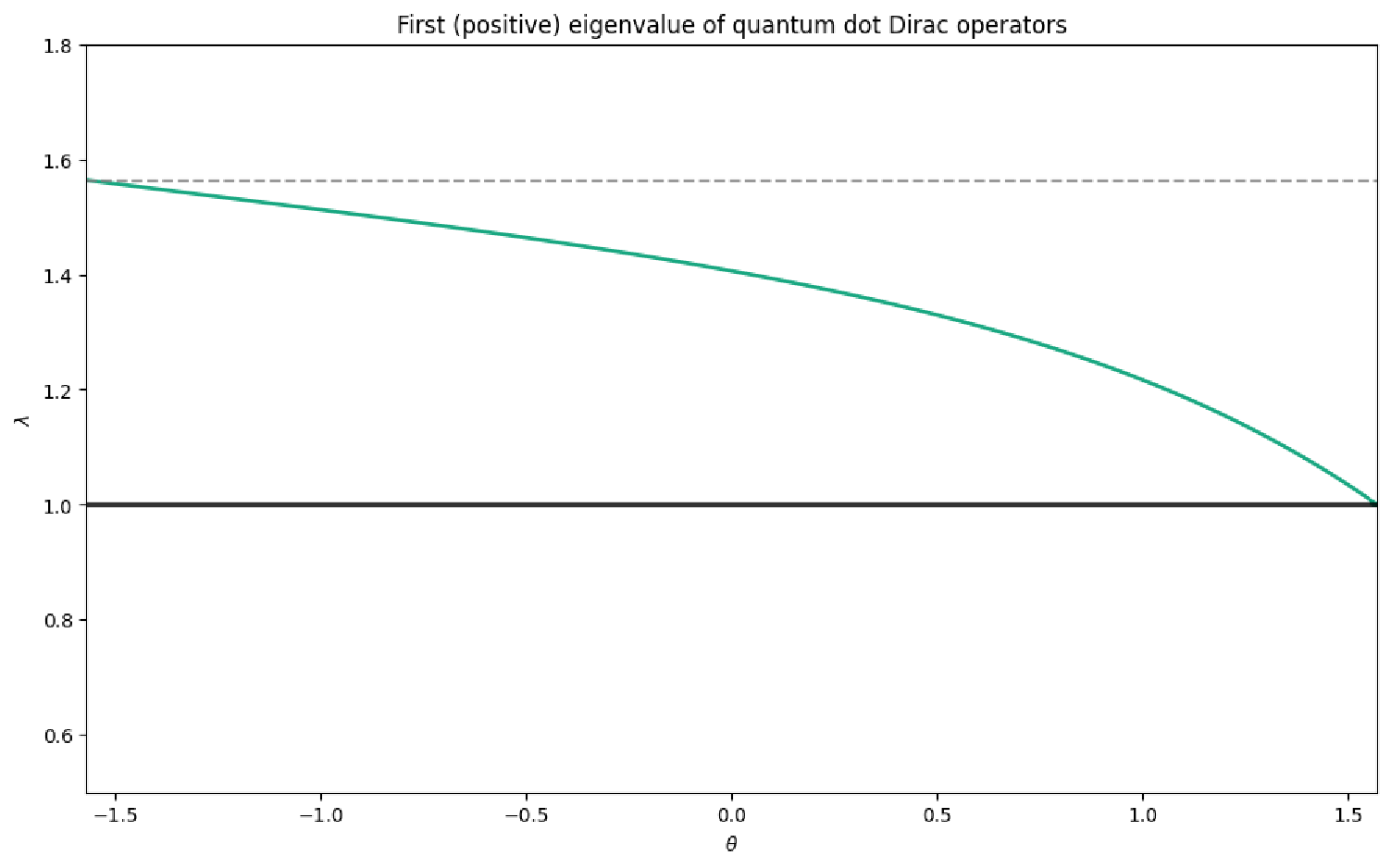}
	\end{minipage}
	\caption{Plot of $a\mapsto \mu_\Omega(a)$ (left) and $\theta\mapsto\lambda_\Omega(\theta)$ (right), where $\Omega$ is a disk $D_R$ of radius $R=2$, and with $m=1$ (represented as a solid black horizontal line in the right picture).
	The horizontal dashed lines represent $\Lambda_{D_R}$ in the left and $\sqrt{\Lambda_{D_R}+m^2}$ in the right.}
	\label{Fig:FirstEigenvalueCurve}
\end{figure}

\begin{theorem}\label{Th:eqiv_conj_point}
Assume that $m\geq0$. Let $\Omega\subset\R^2$ be a bounded domain with $C^2$ boundary and let $D\subset\R^2$ be a disk with the same area as $\Omega$. The following hold:
\begin{itemize}
\item[$(i)$] Given $\theta\in(-\frac \pi 2, \frac \pi 2)$, set 
\begin{equation}\label{def:a_as_theta_conj}
a:=\mu_\Omega^{-1}(\lambda_D(\theta)^2-m^2)>0.
\end{equation} 
If $\mu_\Omega(a)> \mu_D(a)$ then
$\lambda_\Omega(\theta)> \lambda_D(\theta)$.
\item[$(ii)$] Given $a>0$, set 
\begin{equation}\label{def:theta_as_a_conj}
\theta:=\lambda_\Omega^{-1}\big({\textstyle\sqrt{\mu_D(a)+m^2}}\big)
\in{\textstyle(-\frac \pi 2, \frac \pi 2)}.
\end{equation}
If 
$\lambda_\Omega(\theta)> \lambda_D(\theta)$ then
$\mu_\Omega(a)> \mu_D(a)$.
\end{itemize}
\end{theorem}

With this theorem in hand, we will be able to prove the following result. Both its proof and that of \Cref{Th:eqiv_conj_point} are given in \Cref{sec:TheProblem}.

\begin{corollary}\label{l:eqiv_conj_glob}
Conjectures \ref{Conj:Dirac} and \ref{Conj:RobinType} are equivalent.
\end{corollary}

We think these results may help to tackle the Faber-Krahn inequality for the Dirac operator with infinite mass boundary condition (recall that it corresponds to $m=\theta=0$ in \Cref{Conj:Dirac}). 
In view of \Cref{l:eqiv_conj_glob}, the strategy of establishing such inequality by proving \Cref{Conj:RobinType} for the operators $\RR_a$ has some advantages.
First, it reduces to a single PDE ---see \eqref{eq:EigenRobinType}--- while for $\D_{\theta}$ we have a system of two equations ---see \eqref{eq:EigenDirac}.
Moreover, the spectrum of $\RR_a$ is positive and well characterized as min-max levels of a certain Rayleigh quotient which is linear in the parameter $a$ ---see \Cref{thm:PropertiesMu} and \cite[Theorem 1.2]{Duran2025}---, while, so far, the only known variational characterization of $\lambda_\Omega(\theta)$ is given in \cite[Theorem 4]{Antunes2021}, for $\theta = m = 0$, as the zero of a Rayleigh quotient quadratic in the parameter ---see \Cref{rmk:AntunesEtAl} below for more details.
Actually, as we will explain in the following section, we will prove the validity of \Cref{Conj:RobinType} for the operators $\RR_a$ in a certain asymptotic sense, and then use \Cref{Th:eqiv_conj_point} to transfer this information to $\D_{\theta}$.

Note that a disadvantage of our approach is the following: if one wants to prove \Cref{Conj:Dirac} for a given $\theta\in(-\frac \pi 2, \frac \pi 2)$ ---for example $\theta=0$---, a priori one needs to know that \Cref{Conj:RobinType} holds true for all $a>0$, since the change of variables $\theta\mapsto a$ in \eqref{def:a_as_theta_conj} depends on $\Omega$ itself.
Nevertheless, in contrast with what happens for $\D_{\theta}$ (for which the value $\theta=0$ is somehow special since, for example, the spectrum is symmetric with respect to $0$), it does not seem that the problem for $\RR_a$ has a distinguished value of $a$, and this suggests that the parameter $a$ could not have a relevant role in a hypothetical proof of \Cref{Conj:RobinType} ---as happens with the Robin Laplacian, for which the same argument works for all values of the parameter; see~\cite{Daners2006}.

\begin{remark}
\label{rmk:ConjForAllm}
Note that \Cref{Th:eqiv_conj_point} shows that the value of $m\geq 0$ should not play an important role.
Indeed, assume that \Cref{Conj:Dirac} is true for a given value of $m$.
Then, \Cref{Th:eqiv_conj_point}~$(ii)$ would lead the validity of \Cref{Conj:RobinType}, but now we could use \Cref{Th:eqiv_conj_point}~$(i)$ with a different value of $m$ to deduce \Cref{Conj:Dirac} also for this other value of $m$. This entails, in particular, that it is enough to prove \Cref{Conj:Dirac} for $m=0$ to get it for all $m\geq 0$.
\end{remark}

\subsection{Asymptotic regimes} \label{sec:Asymptotic}

Our second main contribution of this work is to show that Conjectures \ref{Conj:Dirac} and \ref{Conj:RobinType} hold true in their asymptotic regimes $\theta\downarrow-\frac{\pi}{2}$ and $a\uparrow+\infty$, respectively, as well as in the regimes $\theta\uparrow\frac{\pi}{2}$ and $a\downarrow 0$ under the additional assumption that $\Omega$ is simply connected. The following two theorems, whose proofs are given in \Cref{sec:asympt_regimes}, contain the precise results in this regard.

\begin{theorem}
	\label{Th:DiscOptimalThetaAsympt}
	Assume that $m\geq0$. Let $\Omega\subset \R^2$ be a bounded domain with $C^2$ boundary and let $D\subset\R^2$ be a disk with the same area as $\Omega$. If $\Omega$ is not a disk,
	then there exists $\theta_0\in(-\frac \pi 2,\frac \pi 2)$ depending on 
	$\Omega$ such that  
$\lambda_\Omega(\theta)>\lambda_D(\theta)$
for all $\theta\in(-\frac{\pi}{2},\theta_0)$. If in addition $\Omega$ is simply connected, then there exists 
$\theta_1\in(-\frac \pi 2,\frac \pi 2)$ depending on 
	$\Omega$ such that  
$\lambda_\Omega(\theta)>\lambda_D(\theta)$ for all $\theta\in(\theta_1,\frac{\pi}{2})$.
\end{theorem}

\begin{theorem}
	\label{thm:VerifyConjLimits}
	Let $\Omega\subset \R^2$ be a bounded domain with $C^2$ boundary and let $D\subset\R^2$ be a disk with the same area as $\Omega$. If $\Omega$ is not a disk,
	then there exists $a_0>0$ depending on 
	$\Omega$ such that  
$\mu_\Omega(a)>\mu_D(a)$
for all $a\in(a_0,+\infty)$. If in addition $\Omega$ is simply connected, then there exists 
$a_1>0$ depending on $\Omega$ such that  
$\mu_\Omega(a)>\mu_D(a)$ for all $a\in(0,a_1)$.
\end{theorem}

The cases $\theta\downarrow-\frac{\pi}{2}$ and $a\uparrow+\infty$ described in \Cref{Th:DiscOptimalThetaAsympt,thm:VerifyConjLimits} will easily follow from the Faber-Krahn inequality for the Dirichlet Laplacian, since 
\begin{equation}
\lim_{\theta\downarrow-\frac{\pi}{2}}\lambda_\Omega(\theta)= \sqrt{\Lambda_\Omega+m^2}
\quad\text{and}\quad
\lim_{a\uparrow+\infty}\mu_\Omega(a)=\Lambda_\Omega;
\end{equation}
see \Cref{thm:PropertiesLambda} and \Cref{thm:PropertiesMu}~$(ii)$, respectively. 
We mention that, in view of the correspondence described in \cite[Section~1.3]{DuranMas2024}, the analogous result for $\theta\downarrow-\frac{\pi}{2}$ in the three-dimensional setting is  established in \cite[Corollary 1.6]{Mas2022}. Moreover, the result for $a\uparrow+\infty$ is already pointed out at the end of \cite[Section~1.4]{Duran2025}.

However, the proof for the other asymptotic regimes of 
$\theta$ and $a$ in \Cref{Th:DiscOptimalThetaAsympt,thm:VerifyConjLimits} will be more involved, and should be considered the main contribution of this section. The difficulty comes from the fact that 
\begin{equation}
\lim_{\theta\uparrow\frac{\pi}{2}}\lambda_\Omega(\theta)= m
\quad\text{and}\quad
\lim_{a\downarrow 0}\mu_\Omega(a)=0
\end{equation}
independently of the shape of $\Omega$;
see again \Cref{thm:PropertiesLambda} and \Cref{thm:PropertiesMu}~$(ii)$, respectively.
Regarding the case $a\downarrow 0$,  
our approach to prove the last statement in \Cref{thm:VerifyConjLimits} will be to analyze the slope of the function $\mu_\Omega$ when departing from the origin. Since $\lim_{a\downarrow 0}\mu_\Omega(a)=0$, this slope is given by $\lim_{a\downarrow 0} \mu_\Omega(a)/a$. In \Cref{prop:muPrimeOrigin} below, we show that this slope is related to the best constant in the embedding of a Hardy space into a Bergman space of holomorphic functions in $\Omega$ ---see \eqref{def:S_Omega} and \eqref{eq:slope_as_inf_quotient1} below. This fact, once combined through a conformal mapping with a well-known inequality for holomorphic functions in the unit disk, will lead to a key estimate of the slope ---see \eqref{eq:slope_as_inf_quotient2}--- from which the case $a\downarrow 0$ will be deduced. Then, the case 
$\theta\uparrow\frac{\pi}{2}$ will be a consequence of the case $a\downarrow 0$ and \Cref{Th:eqiv_conj_point}~$(i)$.
Let us finally emphasize that, again in view of the correspondence described in \cite[Section 1.3]{DuranMas2024}, the last statement in \Cref{Th:DiscOptimalThetaAsympt} gives a positive answer to the problem stated at the end of the paragraph below \cite[Theorem 1.7]{Mas2022} for simply connected domains in the two-dimensional setting.

Before stating \Cref{prop:muPrimeOrigin}, we need to introduce some notation. Firstly, set
\begin{equation} \label{def:E}
    E(\Omega) := \{u\in L^2(\Omega):\,
    \partial_{\bar z} u\in L^2(\Omega)\text{ and }u\in L^2(\partial\Omega)\}.
\end{equation}
Recall from \cite[Lemma 15]{Antunes2021} that if $u$ and $\partial_{\bar z} u$ belong to $L^2(\Omega)$ then the trace of $u$ belongs to $H^{-1/2}(\partial\Omega)$ ---analogous results were shown previously in \cite[Lemma 2.3]{Benguria2017Self} and in \cite[Proposition~2.1]{Vega2016} in the two and three-dimensional settings, respectively. The Hilbert space $E(\Omega)$ defined in~\eqref{def:E} is precisely made by those functions $u\in L^2(\Omega)$ such that  
$\partial_{\bar z} u\in L^2(\Omega)$ and whose traces  actually belong to the smaller space $L^2(\partial\Omega)$; see \cite[Section~3.1]{Duran2025} for a detailed study.

Secondly, the above-mentioned embedding of a Hardy space into a Bergman space of holomorphic functions in $\Omega$ refers to the quantity
\begin{equation}\label{def:S_Omega}
S_\Omega:=\inf_{u\in E(\Omega)\setminus\{0\}:\,
\partial_{\bar z}u=0\text{ in }\Omega}
\frac{\int_{\partial\Omega}|u|^2}{\int_{\Omega}|u|^2}.
\end{equation}

With the definitions of $E(\Omega)$ and $S_\Omega$ in hand, we are ready to state the key result, proven in \Cref{sec:asympt_regimes}, that we will use to show the last statement in \Cref{thm:VerifyConjLimits}.

\begin{proposition}\label{prop:muPrimeOrigin}
Let $\Omega\subset\R^2$ be a bounded domain with $C^2$ boundary. Then,
\begin{itemize}
\item[$(i)$] 
the infimum in \eqref{def:S_Omega} is attained. Actually, there exists a nonzero $u_\Omega\in H^1(\Omega)\subset E(\Omega)$ with
$\partial_{\bar z}u_\Omega=0$ in $\Omega$ and such that
\begin{equation}
S_\Omega
=\frac{\int_{\partial\Omega}|u_\Omega|^2}{\int_{\Omega}|u_\Omega|^2}.
\end{equation}
In addition, any minimizer $u$ of \eqref{def:S_Omega} satisfies
\begin{equation}\label{eq:EL_u_Omega}
S_\Omega\int_{\Omega}u\,\overline v
=\int_{\partial\Omega}u\,\overline v
\quad\text{for all $v\in E(\Omega)$ with
$\partial_{\bar z}v=0$ in $\Omega$}.
\end{equation}
\item[$(ii)$] Furthermore, it holds that
\begin{equation}\label{eq:slope_as_inf_quotient1}
\lim_{a\downarrow 0} \frac{\mu_\Omega(a)}{a}
=S_\Omega.
\end{equation}
\end{itemize}
As a consequence of $(i)$ and $(ii)$, if in addition 
$\Omega$ is simply connected, then 
\begin{equation}\label{eq:slope_as_inf_quotient2}
\lim_{a\downarrow0} \frac{\mu_\Omega(a)}{a} = S_\Omega\geq 2 \sqrt{\frac{\pi}{|\Omega|}}
\end{equation}
and the equality holds if and only if $\Omega$ is a disk.
\end{proposition}
 
\Cref{prop:muPrimeOrigin}~$(i)$ is already proven in \cite[Proposition~2.14]{Duran2025} using the Riesz-Fr\'echet theorem and the spectral theorem. However, in \cite{Duran2025} it is not shown that there exists a minimizer belonging to $H^1(\Omega)$. 
This regularity result will follow from an alternative proof of $(i)$ that we shall provide, which is based on $(ii)$ exploiting the regularity of the eigenfunctions of $\RR_a$ of eigenvalue $\mu_\Omega(a)$.

\subsection{Quantum dots with negative mass} \label{subsec:neg_mass}

Recall that, up to now, we have restricted ourselves to the range $\theta \in (-\frac \pi 2,\frac {\pi}{2})$ for $m\geq 0$. 
Now we would like to study the range $\theta \in (\frac \pi 2,\frac {3\pi}{2})$ for $m\geq 0$ which, in view of \Cref{sec:Invariance}, is equivalent to study the range $\theta \in (-\frac \pi 2,\frac {\pi}{2})$ for $m<0$.
In order to highlight the dependence on $m\in\R$ of the  operator defined in \eqref{def:Dirac_op_theta}, within this section let us denote $\D_\theta$ by $\D_\theta(m)$. Namely,
\begin{equation}
	\begin{split}
		\mathrm{Dom}\big(\D_\theta(m)\big) &:= \big\{ \varphi \in H^1(\Omega)^2: \, \varphi =(\cos\theta\,\sigma\cdot \tau+\sin\theta\, \sigma_3) \varphi  \,\text{ in } H^{1/2}(\partial \Omega)^2 \big\},\\
		\D_\theta(m)\varphi &:= 
		(-i\sigma\cdot\nabla+m\sigma_3)\varphi \quad\text{for all 
			$\varphi\in\mathrm{Dom}\big(\D_\theta(m)\big).$}
	\end{split}
\end{equation}

To motivate the shape optimization problem considered in this section, let us first take a look at the situation when $m<0$ and $\Omega$ is a disk of radius $R$, denoted by $D_R$.
In this case, one can perform the same analysis as for $m\geq 0$ (see \Cref{rmk:RangeConjectureQD}), explicitly finding the eigenfunctions of the operator and deriving the implicit eigenvalue equations. 
If one plots the eigenvalues of $\D_\theta(m)$ as functions of $\theta$, the result is the one contained in \Cref{Fig:EVCurvesQDDirac(NegativeMass)}.

\begin{figure}[h!]
	\includegraphics[width=0.95\textwidth]{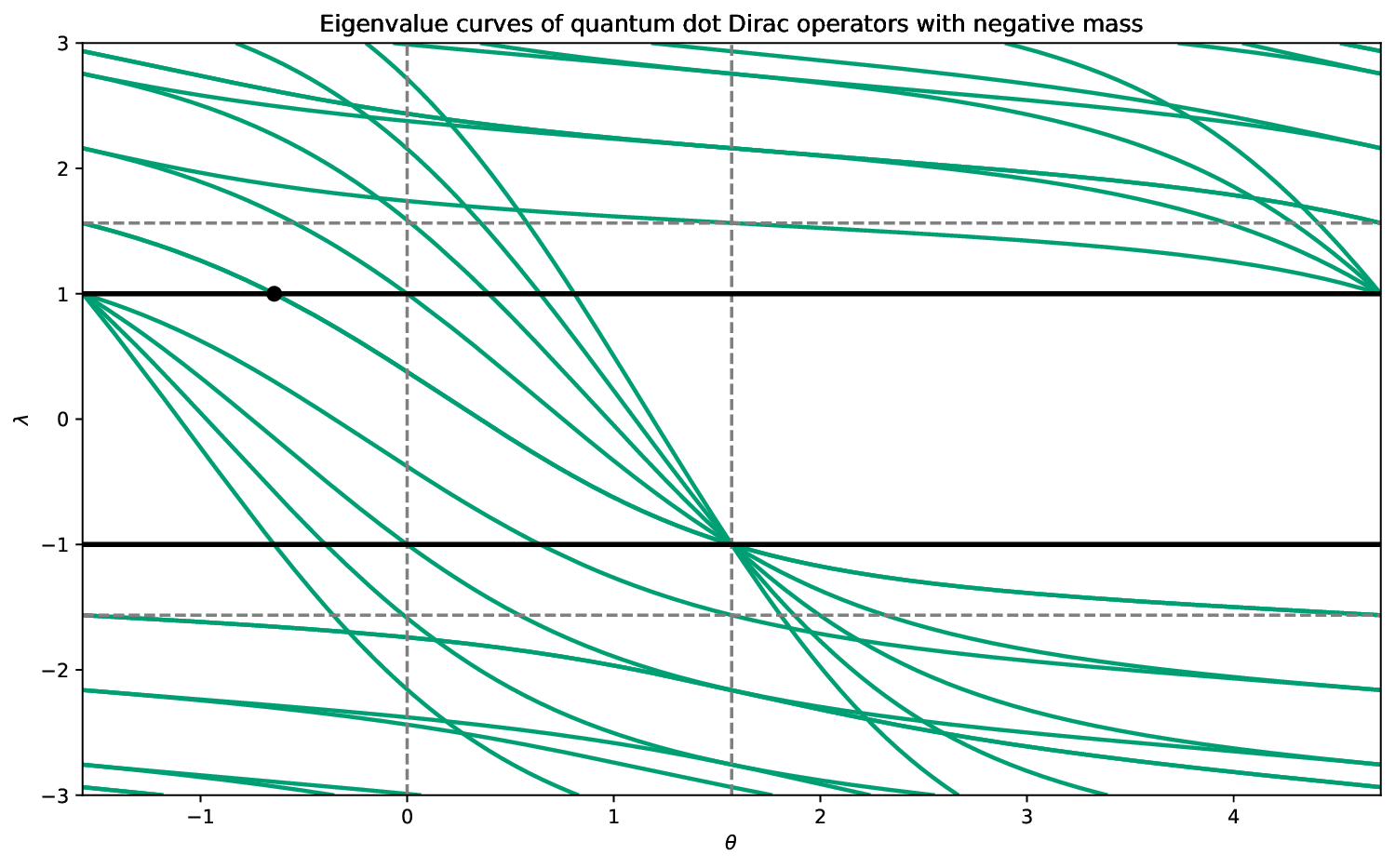}
	\caption{Eigenvalue curves of the quantum dot Dirac operators $\D_{\theta}(m)$ for the disk $D_R$ of radius $R=2$ and with $m=-1$.
	We have included two solid horizontal black lines to highlight the location of $\pm m$.
		The dashed vertical lines at $\theta=0$ and $\theta=\pi/2$ help to locate, respectively, the infinite mass and zigzag cases.
		The horizontal dashed lines represent the values $\pm \sqrt{\Lambda_{D_R}+m^2}$. The black dot illustrates the first crossing point in $(-\frac \pi 2,\frac {\pi}{2})$ between an eigenvalue curve of $\D_\theta(m)$ and the level set $|m|$, which is studied in \Cref{thm:shape_opt_neg_mass}.}
	\label{Fig:EVCurvesQDDirac(NegativeMass)}
\end{figure}

Note that, by comparing this with \Cref{Fig:EVCurvesQDDiracFullRange}, one can see the unitary equivalence of $\D_\theta(m)$ and $-\D_{\pi-\theta}(-m)$, and that $\lambda \in \R$ is an eigenvalue of $\D_\theta(m)$ if and only if  $-\lambda$ is an eigenvalue of $\D_{\pi-\theta}(m)$; this is proved in detail in \Cref{sec:Invariance}.

Let us return for a moment to \Cref{Fig:EVCurvesQDDiracFullRange} with $m\geq 0$.
If \Cref{Conj:Dirac} holds true and we plotted $\theta \mapsto \lambda_\Omega(\theta)$ for any other simply connected $C^2$ domain $\Omega$ with the same area as $D_R$, we would obtain, in $(-\frac \pi 2,\frac {\pi}{2})$, a curve lying above $\lambda_{D_R}(\theta)$. From \Cref{Th:DiscOptimalThetaAsympt} we actually know that this is the case at least for $\theta$ close to $\pm \frac{\pi}{2}$.
Now, if we consider the natural smooth continuation for $\theta \in (-\frac \pi 2,\frac {3\pi}{2})$ of $\lambda_{D_R}(\theta)$ and $\lambda_\Omega(\theta)$, one might expect that, in the interval $(\frac \pi 2,\frac {3\pi}{2})$,  the curve $\lambda_{\Omega}(\theta)$ would lie below the curve $\lambda_{D_R}(\theta)$.
Therefore, it is natural to conjecture that $\lambda_{D_R}(\theta)$ will cross the level set $-m$ ``later'' (i.e., for a greater value of 
$\theta$) than any other curve $\lambda_{\Omega}(\theta)$.\footnote{In general, for any level set $\ell$, if it happened that $\lambda_{D_R} > \lambda_{\Omega}$ in $(\frac \pi 2,\frac {3\pi}{2})$, if $\lambda_{D_R}(\theta_{D_R}) = \ell = \lambda_{\Omega}(\theta_\Omega)$ for some $\theta_{D_R},\theta_\Omega \in (\frac \pi 2,\frac {3\pi}{2})$, and if the curves were monotone, then it would hold that $\theta_\Omega<\theta_{D_R}$.}
Thus, a natural quantity to optimize is the largest value of $\theta$ at which an eigenvalue curve crosses the level set $-m$.
More formally, we seek the largest $\theta$ for which $-m$ is an eigenvalue of $\D_\theta(m)$, and we want to optimize such $\theta$ depending on $\Omega$; recall that here we are considering $m\geq 0$.

Taking into account the unitary equivalence between $\D_\theta(m)$ and $-\D_{\pi-\theta}(-m)$, this is the same as asking, for $m<0$, which is the first (i.e., smallest) $\theta\in (-\frac \pi 2,\frac {\pi}{2})$ for which an eigenvalue curve of $\D_\theta(m)$ \emph{crosses} the level set $|m|$; this crossing point is illustrated in \Cref{Fig:EVCurvesQDDirac(NegativeMass)}. 
In other words, we seek the smallest  $\theta\in (-\frac \pi 2,\frac {\pi}{2})$ for which $|m|$ is an eigenvalue of $\D_\theta(m)$.
Therefore, we are concerned with the eigenvalue problem
\begin{equation}\label{eq:neg_mass_dirac_ev|m|}
	\begin{cases}
		\varphi\in\Dom\big(\D_\theta(m)\big), & \\
		\D_\theta(m)\varphi=|m|\varphi & \text{in } L^2(\Omega)^2. 
	\end{cases}
\end{equation}
We look for the smallest $\theta$ for which \eqref{eq:neg_mass_dirac_ev|m|} has a nonzero solution and we want to find the domain $\Omega$ which makes such $\theta$ as small as possible (under area constraint). 
Our third main result in this paper is the following theorem, which asserts that among all bounded simply connected $C^2$ domains with prescribed area, such $\theta$ is the closest to $-\frac \pi 2$ if and only if $\Omega$ is a disk.

\begin{theorem}\label{thm:shape_opt_neg_mass}
Assume that $m<0$. Let $\Omega\subset\R^2$ be a bounded domain with $C^2$ boundary. Then
\begin{equation}\label{eq:opt_theta_neg_mass_S_Omega}
\begin{split}
\min\big\{\theta\in{\textstyle(-\frac \pi 2, \frac \pi 2)}:\,\eqref{eq:neg_mass_dirac_ev|m|}\text{ has a nonzero solution}\big\}
=\vartheta^{-1}\bigg(\frac{2|m|}{S_\Omega}\bigg),
\end{split}
\end{equation}
where $\vartheta:{\textstyle(-\frac \pi 2, \frac \pi 2)}\to(0,+\infty)$ is defined by
$\vartheta(\theta):={\textstyle\frac{1-\sin\theta}{\cos\theta}}$ and $S_\Omega$ is defined in \eqref{def:S_Omega}.
As a consequence, if in addition 
$\Omega$ is simply connected, then 
\begin{equation}\label{eq:shape_opt_theta_neg_mass_S_Omega}
\begin{split}
\min\big\{\theta\in{\textstyle(-\frac \pi 2, \frac \pi 2)}:\,\eqref{eq:neg_mass_dirac_ev|m|}\text{ has a nonzero solution}\big\}
\geq\vartheta^{-1}\bigg(|m| \sqrt{\frac{|\Omega|}{\pi}}\bigg)
\end{split}
\end{equation}
and the equality holds if and only if $\Omega$ is a disk.
\end{theorem}

The proof of this result, which strongly relies on \Cref{prop:muPrimeOrigin}, is carried out in \Cref{sec:QD_neg_mass}.
As a consequence of this and the unitary equivalence between $\D_\theta(m)$ and $-\D_{\pi-\theta}(-m)$, we obtain the equivalent result in the range $(\frac \pi 2, \frac {3\pi} {2})$ and for $m>0$ mentioned before.
It refers to the eigenvalue problem
\begin{equation}\label{eq:ev_prob_QD_other_branch}
	\begin{cases}
		\varphi\in\Dom\big(\D_\theta(m)\big), & \\
		\D_\theta(m)\varphi=-m\varphi & \text{in } L^2(\Omega)^2,
	\end{cases}
\end{equation}
under the assumptions $\theta\in{\textstyle(\frac \pi 2, \frac {3\pi} {2})}$ and $m>0$.

\begin{corollary}
	\label{thm:shape_opt_neg_mass_unit_equiv}
	Assume that $m>0$. Let $\Omega\subset\R^2$ be a bounded domain with $C^2$ boundary. Then
	\begin{equation}
		\begin{split}
			\max\big\{\theta\in{\textstyle(\frac \pi 2, \frac {3\pi} {2})}:\,\eqref{eq:ev_prob_QD_other_branch}\text{ has a nonzero solution}\big\}
			=\pi-\vartheta^{-1}\bigg(\frac{2m}{S_\Omega}\bigg),
		\end{split}
	\end{equation}
	where $\vartheta$ and $S_\Omega$ are defined in \eqref{def:vartheta_function} and \eqref{def:S_Omega}, respectively.
	As a consequence, if in addition 
	$\Omega$ is simply connected, then 
	\begin{equation}
		\begin{split}
			\max\big\{\theta\in{\textstyle(\frac \pi 2, \frac {3\pi} {2})}:\,\eqref{eq:ev_prob_QD_other_branch}\text{ has a nonzero solution}\big\}
			\leq\pi-\vartheta^{-1}\bigg(m \sqrt{\frac{|\Omega|}{\pi}}\bigg)
		\end{split}
	\end{equation}
	and the equality holds if and only if $\Omega$ is a disk.
\end{corollary}

\section{Equivalence of conjectures} \label{sec:TheProblem}

Throughout this section, we will only consider $\theta\in(-\frac \pi 2,\frac \pi 2)$ and $m\geq 0$ in \eqref{def:Dirac_op_theta}. We will first show how the nonnegative eigenvalues of quantum dot Dirac operators give rise to eigenvalues of $\overline\partial$-Robin Laplacians, and vice versa. Then, we will recall some qualitative properties of $\mu_\Omega$ proven in \cite{Duran2025}, and we will show how to translate them to the Dirac setting. Finally, we will give the proofs of \Cref{Th:eqiv_conj_point} and \Cref{l:eqiv_conj_glob}.

\subsection{Connection between quantum dot Dirac operators and $\overline\partial$-Robin Laplacians}\label{ss:CbqdDoRL}

Given $\theta\in(-\frac \pi 2,\frac \pi 2)$, assume that $\varphi$ solves the eigenvalue problem
\begin{equation}\label{eq:ev_prob_QD}
\begin{cases}
\varphi\in\Dom(\D_\theta), & \\
\D_\theta\varphi=\lambda\varphi & \text{in } L^2(\Omega)^2
\end{cases}
\end{equation}  
for some $\lambda\geq0$.
Writing the equation
$\D_\theta\varphi=\lambda\varphi$ and the boundary condition of $\varphi\in\Dom(\D_\theta)$ in terms of its components $\varphi= (u,v)^\intercal$ (with $u:\Omega\to\C$ and $v:\Omega\to\C$) we get
\begin{equation} \label{eq:EigenDirac}
    \begin{cases}
        -2i \partial_z v = (\lambda-m) u & \text{in } L^2(\Omega), \\
        -2i \partial_{\bar z} u = (\lambda+m)v & \text{in } L^2(\Omega), \\
        v = i \frac{1-\sin \theta}{\cos \theta} \nu u & \text{in } H^{1/2}(\partial \Omega).
    \end{cases}
\end{equation}

As a first step, let us show that if $\varphi$ is not identically zero then $\lambda>m$. This means that, in order to study the nonnegative eigenvalues $\lambda$ of $\D_\theta$ for $\theta\in(-\frac \pi 2,\frac \pi 2)$ and $m\geq 0$, we can always assume that $\lambda>m$. This is the purpose of the next lemma, whose proof is based on \eqref{eq:EigenDirac} and the arguments used in the proof of \cite[Proposition~3.2]{VDBosch2017}.

\begin{lemma}\label{lm:lambda>m}
Let $\theta\in(-\frac \pi 2,\frac \pi 2)$, $m\geq0$, and $\lambda\geq0$. Assume that there exists 
$\varphi\in\Dom(\D_\theta)\setminus\{0\}$ which solves \eqref{eq:ev_prob_QD}. Then $\lambda>m$.
\end{lemma}

\begin{proof}
The proof will follow by a contradiction argument. As before, we write $\varphi= (u,v)^\intercal$. Multiplying the first equation in \eqref{eq:EigenDirac} by $\overline u$, and taking conjugates in the second equation and multiplying it by 
$v$, we get
\begin{equation}
-2i \overline u\partial_z v = (\lambda-m) |u|^2
\quad\text{and}\quad
2i v\partial_{z} \overline u = (\lambda+m)|v|^2
\quad\text{in $\Omega$.}
\end{equation}
If we subtract the second identity to the first one and then we integrate in $\Omega$, using the divergence theorem, the third equation in \eqref{eq:EigenDirac}, and the fact that $\overline\nu\nu=|\nu|^2=1$, we obtain
\begin{equation}\label{ineq:lambda>m}
\begin{split}
(\lambda-m) \int_\Omega|u|^2-(\lambda+m)\int_\Omega|v|^2
&=-2i\int_\Omega \big(\overline u\partial_z v + v\partial_{z} \overline u  \big)
=-2i\int_\Omega \partial_z(\overline uv)\\
&=-i\int_{\partial\Omega} \overline\nu\overline uv
= {\frac{1-\sin \theta}{\cos \theta}}\int_{\partial\Omega} |u|^2\\
&= {\frac{1-\sin \theta}{2\cos \theta}}\int_{\partial\Omega} {|u|^2}
+{\frac{\cos \theta}{2(1-\sin \theta)}}\int_{\partial\Omega} {|v|^2}\geq0,
\end{split}
\end{equation}
because $\theta\in(-\frac \pi 2,\frac \pi 2)$. 

Recall that we are assuming $m\geq0$ and $\lambda\geq0$. In order to prove that $\lambda>m$ if $\varphi$ does not vanish identically, it suffices to assume that $m\geq\lambda\geq0$ and reach a contradiction.
If $m\geq\lambda\geq0$ then the left-hand side of $\eqref{ineq:lambda>m}$ is nonpositive and, thus, we deduce that
\begin{equation}\label{ineq:lambda>m_2}
\begin{split}
(\lambda-m) \int_\Omega|u|^2-(\lambda+m)\int_\Omega|v|^2= {\frac{1-\sin \theta}{2\cos \theta}}\int_{\partial\Omega} {|u|^2}
+{\frac{\cos \theta}{2(1-\sin \theta)}}\int_{\partial\Omega} {|v|^2}=0,
\end{split}
\end{equation}
which yields $(\lambda-m) \int_\Omega|u|^2=(\lambda+m)\int_\Omega|v|^2=0$ and $u=v=0$ on $\partial\Omega$.
At this point we distinguish two cases: $m>0$ and $m=0$. On the one hand, if
$m>0$ then $\lambda+m>0$ and, thus, $\int_\Omega|v|^2=0$. That is, $v=0$ in $\Omega$. Therefore, the second equation in \eqref{eq:EigenDirac} leads to $\partial_{\bar z} u=0$ in $\Omega$. Since $u=0$ on $\partial\Omega$, we conclude that $u=0$ in $\Omega$ by the unique continuation principle for holomorphic functions. This yields 
$\varphi= (u,v)^\intercal=0$ in $\Omega$, which contradicts the assumption in the statement of the lemma. On the other hand, if $m=0$, since we are assuming that $m\geq\lambda\geq0$, we get $\lambda=m=0$. Therefore, the first and second equations in \eqref{eq:EigenDirac} lead to 
$\partial_{\bar z} \overline v=\overline{\partial_z v}=0$ and 
$\partial_{\bar z} u=0$ in $\Omega$, respectively. Since $u=\overline v=0$ on $\partial\Omega$, again by unique continuation we conclude that $\varphi= (u,v)^\intercal=0$ in $\Omega$, leading to a contradiction.
\end{proof}

In view of \Cref{lm:lambda>m}, from now on we will assume that $\lambda>m\geq0$ in \eqref{eq:ev_prob_QD}. The following argumentation, which is reminiscent of \cite[Remark 5]{Antunes2021},
shows how the eigenvalue equation associated to $\D_\theta$ rewrites in terms of the Laplace differential operator. As before, given $\theta\in(-\frac \pi 2,\frac \pi 2)$, assume that $\varphi$ solves the eigenvalue problem  \eqref{eq:ev_prob_QD}. 
On the one hand, since $4\partial_z\partial_{\bar z}
=\Delta$, applying $-2i\partial_{ z}$ in the distributional sense to the second equation in \eqref{eq:EigenDirac} and then using the first one we obtain
\begin{equation}
    -\Delta u = (\lambda^2-m^2)u \quad \text{in } L^2(\Omega).
\end{equation}
On the other hand, since $\Dom(\D_\theta)\subset H^1(\Omega)^2$, from the second equation in \eqref{eq:EigenDirac} we actually see that $-2i\partial_{\bar z} u = (\lambda+m)v\in H^1(\Omega)$. If we now take boundary traces and then we multiply both sides of the equality by $\overline\nu$, we deduce that
$-2i\overline\nu\partial_{\bar z} u = (\lambda+m)\overline\nu v\in H^{1/2}(\partial\Omega)$ which, combined with the third equation in \eqref{eq:EigenDirac} and the fact that $\overline\nu\nu=|\nu|^2=1$, leads to
\begin{equation}
    2\overline\nu\partial_{\bar z} u + (\lambda+m)\frac{1-\sin\theta}{\cos\theta} u=0 \quad \text{in } H^{1/2}(\partial\Omega).
\end{equation}
In conclusion, if 
$\theta\in(-\frac \pi 2,\frac \pi 2)$ and 
$\varphi= (u,v)^\intercal\in\Dom(\D_\theta)$ solves 
$\D_\theta\varphi=\lambda\varphi$ for some $\lambda>m$, then $u$ solves the eigenvalue problem 
\begin{equation} \label{eq:EigenRobinType}
    \begin{cases}
        u,\partial_{\bar z} u\in H^1(\Omega), & \\
        -\Delta u = \mu u & \text{in } L^2(\Omega), \\
        2 \overline \nu \partial_{\bar z} u + au = 0 & \text{in } 
        H^{1/2}(\partial \Omega),
    \end{cases}
\end{equation}
with
\begin{equation} \label{eq:ParamDiracRob}
    \mu := \lambda^2-m^2>0 
    \quad\text{and}\quad
    a:=(\lambda+m)\frac{1-\sin\theta}{\cos\theta}>0.
\end{equation}
Note that \eqref{eq:EigenRobinType} is precisely the eigenvalue equation for the 
$\overline\partial$-Robin Laplacian $\RR_a$ ---recall the definition in \eqref{eq:RodzinLaplacian}. The boundary condition in \eqref{eq:EigenRobinType} is also referred in the literature as Cauchy-Riemann
oblique boundary condition.

The relation \eqref{eq:ParamDiracRob} motivates to introduce the following functions, which will be used in the sequel. The first one is the  smooth, strictly decreasing, and bijective function 
\begin{equation}\label{def:vartheta_function}
\begin{split}
\vartheta:{\textstyle(-\frac \pi 2, \frac \pi 2)}&\to(0,+\infty)\\
\theta&\mapsto{\textstyle\frac{1-\sin\theta}{\cos\theta}},
\end{split}
\end{equation}
and the second one, which describes the relation between the pairs
$(\theta,\lambda)$ and $(a,\mu)$ in  \eqref{eq:ParamDiracRob}, is 
\begin{equation}\label{def:T_function}
\begin{split}
T:{\textstyle(-\frac \pi 2, \frac \pi 2)}\times(m,+\infty)
&\to(0,+\infty)\times(0,+\infty)\\
(\theta,\lambda)&\mapsto \big((\lambda+m)\vartheta(\theta),\lambda^2-m^2\big).
\end{split}
\end{equation}
With these definitions in hand, \eqref{eq:ParamDiracRob} rewrites as 
$(a,\mu):=T(\theta,\lambda)$. Note also that, given $\mu>0$ and $a>0$, if we take $\lambda:=\sqrt{\mu+m^2}>m$ then there exists a unique $\theta\in(-\frac \pi 2, \frac \pi 2)$ satisfying $a=(\lambda+m)\vartheta(\theta)$, since the function $\vartheta$ is bijective from $(-\frac \pi 2, \frac \pi 2)$ to $(0,+\infty)$. Therefore, the function $T$ is bijective and its inverse is given by
\begin{equation}\label{def:T-1_function}
\begin{split}
T^{-1}:
(0,+\infty)\times(0,+\infty)&\to{\textstyle(-\frac \pi 2, \frac \pi 2)}\times(m,+\infty)\\
(a,\mu)&\mapsto 
\Big(\vartheta^{-1}\Big({\textstyle\frac{a}{\sqrt{\mu+m^2}+m}}\Big),\sqrt{\mu + m^2}\Big).
\end{split}
\end{equation}

We have seen that a solution to \eqref{eq:ev_prob_QD} yields a solution to \eqref{eq:EigenRobinType}. The reverse implication follows similarly. Assume that $u$ solves \eqref{eq:EigenRobinType} for some 
$\mu>0$ and $a>0$. Then, setting 
\begin{equation} \label{eq:ParamRobDirac}
\begin{split}
    (\theta,\lambda):=T^{-1}(a,\mu)\quad\text{and}\quad
    v:= \frac{-2i}{\lambda+m} \partial_{\bar z} u\in H^1(\Omega),
\end{split}
\end{equation}
we deduce that $u,v\in H^1(\Omega)$ solve \eqref{eq:EigenDirac}, or in other words, that $\varphi:=(u,v)^\intercal$ solves 
\eqref{eq:ev_prob_QD} for $\theta\in(-\frac \pi 2, \frac \pi 2)$ and $\lambda>m$ as in \eqref{eq:ParamRobDirac}. 

In conclusion, the eigenvalue problem \eqref{eq:ev_prob_QD} for the quantum dot Dirac operator with $\theta\in(-\frac \pi 2, \frac \pi 2)$ and $\lambda>m$ is equivalent to the eigenvalue problem \eqref{eq:EigenRobinType} for the $\overline\partial$-Robin Laplacian with $a>0$ and $\mu>0$ under the relation $(a,\mu)=T(\theta,\lambda)$. 
This will allow us to study the first nonnegative eigenvalue $\lambda_\Omega(\theta)$ of $\D_\theta$ through its reformulation in the framework of the first eigenvalue $\mu_\Omega(a)$ of $\RR_a$. 

The next key result (crucial to obtain \Cref{Th:eqiv_conj_point}) shows that indeed $\lambda_\Omega$ and $\mu_\Omega$ are mapped to each other through the function $T$ defined in \eqref{def:T_function}. 
This intuitive fact is not completely obvious because of the following observation: if, for a given $\theta$, one takes two different eigenvalues $\lambda_1$ and $\lambda_2$ of $\D_\theta$ and then, using $T$, one constructs the corresponding eigenvalues $\mu_1$ and $\mu_2$ of the $\overline\partial$-Robin Laplacian, it may happen that $\mu_1\in\sigma(\RR_{a_1})$ and $\mu_2\in\sigma(\RR_{a_2})$ for different parameters $a_1$ and $a_2$ ---recall that $a$ in \eqref{eq:ParamDiracRob} depends on $\lambda$.

\begin{proposition}\label{lm:1st_to_1st}
	Let $(\theta,\lambda)\in{\textstyle(-\frac \pi 2, \frac \pi 2)}\times(m,+\infty)$ and $(a,\mu)\in(0,+\infty)\times(0,+\infty)$ be such that $T(\theta,\lambda)=(a,\mu)$. Then, 
	$\lambda=\lambda_\Omega(\theta)$ if and only if 
	$\mu=\mu_\Omega(a)$.
\end{proposition}

The proof of this result will be given in \Cref{subsec:PropCurves}.
As will be seen, it is not completely obvious and requires some properties of the eigenvalue curve $a \mapsto \mu_\Omega(a)$ that will be recalled from \cite{Duran2025} or obtained in the next section.

\subsection{Properties of the eigenvalue curves}
\label{subsec:PropCurves}

Here we will recall some properties of the map $a\mapsto \mu_\Omega(a)$ proven in \cite{Duran2025} which will be key ingredients in the present work. 
Next, we will establish other properties needed to obtain \Cref{lm:1st_to_1st}.
Finally, after proving \Cref{lm:1st_to_1st}, we will transfer the properties of $a\mapsto \mu_\Omega(a)$ to $\theta\mapsto \lambda_\Omega(\theta)$.

To start with, in the next statement we recall the main properties of $a\mapsto \mu_\Omega(a)$ proven in~\cite{Duran2025}.
Recall that $\Lambda_\Omega$ denotes the first eigenvalue of the self-adjoint realization of the Dirichlet Laplacian in $L^2(\Omega)$.

\begin{theorem} \label{thm:PropertiesMu}
$($\cite[Theorems 1.2 and 1.3 $(i)$]{Duran2025}$)$ 
    Given $a>0$, let $\mu_\Omega(a)=\min\big(\sigma(\RR_a)\big)$. 
    The following hold:
    \begin{enumerate}[label=$(\roman*)$]
    \item Let $E(\Omega)=\{u\in L^2(\Omega):\,
    \partial_{\bar z} u\in L^2(\Omega)\text{ and }u\in L^2(\partial\Omega)\}$, as in \eqref{def:E}. 
    Then,
    \begin{equation}\label{eq:RQ_Rodzin_mu}
    \mu_\Omega(a) = \inf_{u\in E(\Omega)\setminus\{0\}}\dfrac{4\int_\Omega |\partial_{\bar z} u|^2 + a\int_{\partial\Omega} |u|^2}{\int_\Omega |u|^2}
\end{equation}
       and the infimum is attained. Furthermore, any minimizer $u_\Omega(a)$ of \eqref{eq:RQ_Rodzin_mu} belongs to 
       $\Dom(\RR_a)\setminus\{0\} \subset E(\Omega)$ and solves  
       $\RR_a u_\Omega(a)=\mu_\Omega(a)u_\Omega(a)$.       
    \item \label{itemC1andIncr} The function $a\mapsto \mu_\Omega(a)$ is continuous, strictly increasing, and bijective from $(0,+\infty)$ to $(0,\Lambda_\Omega)$. \end{enumerate}
\end{theorem}

Before continuing, for the sake of clarity, let us make a couple of remarks.

\begin{remark}\label{rmk:RobinVsRodzin}
	Let us recall the simple computation that motivates the variational formulation of $\mu_\Omega(a)$ stated in \eqref{eq:RQ_Rodzin_mu}.
	Assume that $u$ does not vanish identically and solves \eqref{eq:EigenRobinType} for some $a>0$ and $\mu>0$.  Since $\mu u=-\Delta u = -4\partial_z\partial_{\bar z}u$ in $\Omega$, if we multiply this equality by $\overline u$ and integrate it in $\Omega$, the divergence theorem and the boundary condition in \eqref{eq:EigenRobinType} yield
	\begin{equation}
		\mu\int_\Omega |u|^2
		=-4\int_\Omega \partial_z\partial_{\bar z}u\,\overline u
		=4\int_\Omega |\partial_{\bar z}u|^2
		-\int_{\partial\Omega}2\overline\nu \partial_{\bar z}u\,\overline u
		=4\int_\Omega |\partial_{\bar z}u|^2
		+a\int_{\partial\Omega}|u|^2.
	\end{equation}
	This is the identity which gives rise to \eqref{eq:RQ_Rodzin_mu}.
	Observe also the similitude of the variational formulation of $\mu_\Omega(a)$ in \eqref{eq:RQ_Rodzin_mu} with the one of the first eigenvalue of the Robin Laplacian
	\begin{equation} 
		\begin{cases}
			-\Delta u = \mu u & \text{in } \Omega, \\
			\partial_\nu u + a u = 0 & \text{on } \partial\Omega,
		\end{cases}
	\end{equation}
	which has the same form as in \eqref{eq:RQ_Rodzin_mu} but replacing $\partial_{\bar z} u$ by $\nabla u$ and $E(\Omega)$ by $H^1(\Omega)$ ---this also motivates the name $\overline\partial$-\textit{Robin Laplacian}.
\end{remark}

\begin{remark}
	\label{rmk:AntunesEtAl}
	Recall that, thanks to \Cref{lm:1st_to_1st} (to be proved later), $\lambda_\Omega$ and $\mu_\Omega$ are mapped to each other through the function $T$ defined in \eqref{def:T_function}.  
	From this and the expression \eqref{eq:RQ_Rodzin_mu} of~$\mu_\Omega$ we have a characterization of $\lambda_\Omega(\theta)$.
	Let us compare it with the variational characterization of $\lambda_\Omega(0)$ given in \cite[Theorem~4]{Antunes2021}.
	
	As we mentioned in the introduction, the problem for $\theta=m=0$ was studied in~\cite{Antunes2021}, which is one of our main inspirations.
	In this particular case the relation $(a, \mu) = T(\theta, \lambda)$ given in \eqref{eq:ParamDiracRob} reduces to
	\begin{equation}  \label{eq:RelationInfiniteMassTheta=0}
		\mu = \lambda^2
		\quad\text{and}\quad
		a=\lambda,
	\end{equation}
	and therefore the problem \eqref{eq:EigenRobinType} can be written as 
	\begin{equation}\label{eq:EVequationAntunes}
		\begin{cases}
			-\Delta u = a^2 u & \text{in } L^2(\Omega), \\
			2 \overline \nu \partial_{\bar z} u + a u = 0 & \text{in } 
			H^{1/2}(\partial \Omega).
		\end{cases}
	\end{equation}
	This is what is done in \cite{Antunes2021} to obtain a variational characterization of $\lambda_\Omega(0)$.
	Indeed, \cite[Theorem~4]{Antunes2021} states that $\lambda>0 $ is $\lambda_\Omega(0)$ if and only if $\mathcal{P}(\lambda)=0$, where
	\begin{equation}
		\mathcal{P}(\lambda) := \inf_{u\in E(\Omega)\setminus\{0\}}\dfrac{4\int_\Omega |\partial_{\bar z} u|^2 + \lambda \int_{\partial\Omega} |u|^2 - \lambda^2 \int_\Omega |u|^2}{\int_\Omega |u|^2}.
	\end{equation}
	
	From the point of view of the present work, as a consequence of \eqref{eq:RQ_Rodzin_mu} and the relation \eqref{eq:RelationInfiniteMassTheta=0}, by \Cref{lm:1st_to_1st} we are lead to find $\lambda>0$ such that 
	\begin{equation}
		\lambda^2 = \inf_{u\in E(\Omega)\setminus\{0\}}\dfrac{4\int_\Omega |\partial_{\bar z} u|^2 + \lambda \int_{\partial\Omega} |u|^2}{\int_\Omega |u|^2},
	\end{equation}
	which is precisely $\mathcal{P}(\lambda)=0$.
	
	Note that the parameter $a$ appears in \eqref{eq:EVequationAntunes} both in the equation (as~$a^2$) and the boundary condition (as~$a$).
	Similarly, $\mathcal{P}(\lambda)$ has a linear and a quadratic term in $\lambda$.	
	In our case, it is crucial that we \emph{decouple} the parameter appearing in the PDE and the one appearing in the boundary condition. 
	By doing this, $\mu = \lambda^2 = a^2$ becomes an eigenvalue and we only have a parameter in the boundary condition.
	This is a key point for us, since it allows us to work with a characterization of the eigenvalue as a minimum of a Rayleigh quotient which is \emph{linear} in $a$.
	The price to pay, as we mentioned in the introduction, is that in order to prove \Cref{Conj:Dirac} for a given $\theta$, we would have to prove \Cref{Conj:RobinType} for all values $a>0$.  
\end{remark}

Next, we establish two properties of $a\mapsto \mu_\Omega(a)$ and related functions which will be used later.

\begin{lemma}\label{lm:a/mu_decrease}
The function $a\mapsto\mu_\Omega(a)/a$ is strictly decreasing in $(0,+\infty)$. Moreover, 
\begin{equation}\label{ineq:est_mu/a_per_vol}
\frac{\mu_\Omega(a)}{a}\leq\frac{|\partial\Omega|}{|\Omega|}\quad\text{for all $a>0$.}
\end{equation}
\end{lemma}

\begin{proof}
We first show the claimed monotonicity; the argument will be the same as the one in the proof of \cite[Proposition 33(3)]{Antunes2021}. Assume that $a_2>a_1>0$ and let $u_\Omega(a_1)$ be a minimizer of $\mu_\Omega(a_1)$ as in \Cref{thm:PropertiesMu}~$(i)$ for $a=a_1$. Without loss of generality, we can assume that 
$\|u_\Omega(a_1)\|_{L^2(\Omega)}=1$. Observe also that $\int_\Omega |\partial_{\bar z} u_\Omega(a_1)|^2\neq 0$, since otherwise $$\mu_\Omega(a_1) u_\Omega(a_1) = \RR_{a_1} u_\Omega(a_1) = -\Delta u_\Omega(a_1) = -4\partial_z\partial_{\bar z} u_\Omega(a_1) = 0,$$ contradicting the fact that $\mu_\Omega(a_1)>0$, as stated in \Cref{thm:PropertiesMu}~$(ii)$. Then, from \eqref{eq:RQ_Rodzin_mu} we get
\begin{equation}
\begin{split}
    \frac{a_2}{a_1}\mu_\Omega(a_1) &= \frac{a_2}{a_1}\,4\int_\Omega |\partial_{\bar z} u_\Omega(a_1)|^2 + a_2\int_{\partial\Omega} |u_\Omega(a_1)|^2\\
    &>4\int_\Omega |\partial_{\bar z} u_\Omega(a_1)|^2 + a_2\int_{\partial\Omega} |u_\Omega(a_1)|^2
    \geq \mu_\Omega(a_2).
    \end{split}
\end{equation}
This shows that $\mu_\Omega(a_1)/a_1>\mu_\Omega(a_2)/a_2$, as desired.

The proof of \eqref{ineq:est_mu/a_per_vol} follows by testing \eqref{eq:RQ_Rodzin_mu} with the constant function $u=1$ in $\Omega$.
\end{proof}

\begin{lemma}\label{lm:finvertible}
	The function
	\begin{equation}
		a\mapsto f(a):= \vartheta^{-1}\Big({\textstyle\frac{a}{\sqrt{\mu_\Omega(a)+m^2}+m}}\Big)
	\end{equation}
	is continuous, strictly decreasing, and bijective from $(0,+\infty)$ to 
	$(-\frac \pi 2,\frac \pi 2)$. 
\end{lemma}

\begin{proof}
	Recall from \Cref{thm:PropertiesMu}~$(ii)$ that the function $a\mapsto \mu_\Omega(a)$ is continuous and strictly positive in $(0,+\infty)$. 
	Consequently, the function $f$ is well-defined and continuous in $(0,+\infty)$ ---recall the definition of $\vartheta$ in \eqref{def:vartheta_function}. 
	We will see now that $f$ is strictly decreasing. Note that since $a \mapsto \mu_\Omega(a)/a$ is strictly decreasing thanks to \Cref{lm:a/mu_decrease},
	\begin{equation}\label{eq:fct_theta_a_aux}
		{\textstyle\frac{\sqrt{\mu_\Omega(a)+m^2}+m}{a}}
		={\textstyle\sqrt{\frac{1}{a}(\frac{\mu_\Omega(a)}{a}+\frac{m^2}{a})}+\frac{m}{a}}
	\end{equation} 
	is a strictly decreasing function of $a\in(0,+\infty)$ ---recall that $m\geq0$. 
	Therefore, the function 
	\begin{equation}
		a\mapsto{\textstyle\frac{a}{\sqrt{\mu_\Omega(a)+m^2}+m}}
	\end{equation} 
	is strictly increasing in $(0,+\infty)$. Then, using that 
	$\vartheta^{-1}$ is strictly decreasing, we deduce that
	$a\mapsto f(a)$ is strictly decreasing in $(0,+\infty)$, as claimed.
	
	Finally, from \eqref{eq:fct_theta_a_aux} and \Cref{thm:PropertiesMu}~$(ii)$ (and using also \Cref{lm:a/mu_decrease} for the second limit below) we see that 
	\begin{equation}
		\lim_{a\uparrow+\infty}{\textstyle\frac{\sqrt{\mu_\Omega(a)+m^2}+m}{a}}=0
		\quad\text{and}\quad
		\lim_{a\downarrow 0}{\textstyle\frac{\sqrt{\mu_\Omega(a)+m^2}+m}{a}}=+\infty,
	\end{equation} 
	which easily yields $\lim_{a\uparrow+\infty}f(a)=-\frac \pi 2$ and 
	$\lim_{a\downarrow 0}f(a)=\frac \pi 2$. 
	From this we conclude that the function $a\mapsto f(a)$ is continuous, strictly decreasing, and bijective from $(0,+\infty)$ to 
	$(-\frac \pi 2,\frac \pi 2)$. 	
\end{proof}

With the previous properties of $a\mapsto \mu_\Omega(a)$ and related functions at hand, we can now prove that $\lambda_\Omega$ and $\mu_\Omega$ are mapped to each other through the function $T$ defined in \eqref{def:T_function}.

\begin{proof}[Proof of \Cref{lm:1st_to_1st}]
Let us first prove that, given $a>0$, if $T^{-1}(a,\mu_\Omega(a))=(\theta,\lambda)$ for some $\theta\in{\textstyle(-\frac \pi 2, \frac \pi 2)}$ and $\lambda>m$, then $\lambda=\lambda_\Omega(\theta)$. 
We will prove this statement by contradiction. 
Hence, assume that $T^{-1}(a,\mu_\Omega(a))=(\theta,\lambda)$ but that $\lambda\neq\lambda_\Omega(\theta)$.
Then necessarily $\lambda>\lambda_\Omega(\theta)$, since $\lambda_\Omega(\theta)$ is the first nonnegative eigenvalue of $\D_{\theta}$.
Since $T^{-1}(a,\mu_\Omega(a))=(\theta,\lambda)$, we have that
\begin{equation}
a=(\lambda+m)\vartheta(\theta)
\quad\text{and}\quad
\mu_\Omega(a)=\lambda^2-m^2.
\end{equation}
Next, set $(a^*,\mu^*):=T(\theta,\lambda_\Omega(\theta))$, that is,
\begin{equation}
a^*:=(\lambda_\Omega(\theta)+m)\vartheta(\theta)
\quad\text{and}\quad
\mu^*:=\lambda_\Omega(\theta)^2-m^2.
\end{equation}
Note that since $(a^*,\mu^*)=T(\theta,\lambda_\Omega(\theta))$ and $\lambda_\Omega(\theta)>m$ is an eigenvalue of $\D_\theta$, $\mu^*>0$ is an eigenvalue of $\RR_{a^*}$.
Thus, since $\mu_\Omega(a^*)$ is the smallest eigenvalue of $\RR_{a^*}$, we have that $\mu^*\geq\mu_\Omega(a^*)$. 
Note also that 
\begin{equation}
	\frac{a}{a^*}
	=\frac{\lambda+m}{\lambda_\Omega(\theta)+m}.
\end{equation}
Moreover, since $\lambda>\lambda_\Omega(\theta)$ we also have that $a>a^*$.

We are now ready to reach the contradiction. On the one hand, 
\begin{equation}
\frac{\mu_\Omega(a)}{\mu^*}
=\frac{\lambda^2-m^2}{\lambda_\Omega(\theta)^2-m^2}
=\frac{\lambda-m}{\lambda_\Omega(\theta)-m} \,
\frac{\lambda+m}{\lambda_\Omega(\theta)+m}
=\frac{\lambda-m}{\lambda_\Omega(\theta)-m}\, \frac{a}{a^*},
\end{equation}
which, thanks to the fact that $\lambda>\lambda_\Omega(\theta)>m$, leads to 
\begin{equation}\label{eq:contr1:1st_to_1st}
\frac{a^*}{\mu^*}\frac{\mu_\Omega(a)}{a}
=\frac{\lambda-m}{\lambda_\Omega(\theta)-m}>1.
\end{equation}
On the other hand, using that $a>a^*$, that $a\mapsto\mu_\Omega(a)/a$ is strictly decreasing  by \Cref{lm:a/mu_decrease}, and that $\mu^*\geq\mu_\Omega(a^*)$, we obtain
\begin{equation}\label{eq:contr2:1st_to_1st}
\frac{a^*}{\mu^*}\frac{\mu_\Omega(a)}{a}
<\frac{a^*}{\mu^*}\frac{\mu_\Omega(a^*)}{a^*}
=\frac{\mu_\Omega(a^*)}{\mu^*}\leq1,
\end{equation}
which contradicts \eqref{eq:contr1:1st_to_1st}.

We have proven that, given $a>0$, 
\begin{equation}\label{claim:1st_to_1st}
\text{if $T^{-1}(a,\mu_\Omega(a))=(\theta,\lambda)$ for some 
$\theta\in{\textstyle(-\frac \pi 2, \frac \pi 2)}$ and $\lambda>m$ then $\lambda=\lambda_\Omega(\theta)$.}
\end{equation} 
Since $T$ is bijective, this proves the ``if'' implication of the lemma. 

Let us now address the ``only if'' implication. Assume that, given $\theta\in(-\frac \pi 2,\frac \pi 2 )$, we have 
$T(\theta,\lambda_\Omega(\theta))=(a,\mu)$ for some $a>0$ and $\mu>0$. We want to prove that $\mu=\mu_\Omega(a)$. From \Cref{lm:finvertible}, we see that
there exists $a^*>0$ such that 
\begin{equation}\label{eq:existence_theta}
\theta=\vartheta^{-1}\Big({\textstyle\frac{a^*}{\sqrt{\mu_\Omega(a^*)+m^2}+m}}\Big).
\end{equation}
Hence, from \eqref{eq:existence_theta}, the definition of $T^{-1}$, and \eqref{claim:1st_to_1st}, it follows that 
$T^{-1}(a^*,\mu_\Omega(a^*))=(\theta,\lambda_\Omega(\theta))$. Using that $T$ is bijective, we obtain that
$(a^*,\mu_\Omega(a^*))=T(\theta,\lambda_\Omega(\theta))
=(a,\mu)$ and, therefore, that $a=a^*$ and 
$\mu=\mu_\Omega(a^*)=\mu_\Omega(a)$, as desired.
\end{proof}

As we mentioned, combining \Cref{thm:PropertiesMu}~$(ii)$ with the mapping $T$ defined in \eqref{def:T_function} one can transfer the qualitative properties of $\mu_\Omega$ to
$\lambda_\Omega$, as the following result shows. Let us also mention that these properties of $\lambda_\Omega$ can be proven without appealing to $\mu_\Omega$ but using instead boundary integral operators and perturbation theory at the level of resolvents ---as it is done in the three-dimensional framework in \cite[Section 3]{Mas2022}.

\begin{proposition}\label{thm:PropertiesLambda}
	Given $\theta\in(-\frac \pi 2, \frac \pi 2)$, 
	let $\lambda_\Omega(\theta)$ be as in \eqref{def:1stDiracEigenvalue}. Then, the function 
	$\theta\mapsto \lambda_\Omega(\theta)$ is continuous, strictly decreasing, and bijective from $(-\frac \pi 2, \frac \pi 2)$ to $\big(m,\sqrt{\Lambda_\Omega+m^2}\big)$.
\end{proposition}

\begin{proof}
Recall from \Cref{lm:finvertible} that the function
\begin{equation}
    a\mapsto f(a):= \vartheta^{-1}\Big({\textstyle\frac{a}{\sqrt{\mu_\Omega(a)+m^2}+m}}\Big)
\end{equation} is continuous, strictly decreasing, and bijective from $(0,+\infty)$ to 
$(-\frac \pi 2,\frac \pi 2)$. 
In particular, its inverse function
$\theta\mapsto f^{-1}(\theta)$ is well-defined, continuous, strictly decreasing, and bijective from $(-\frac \pi 2,\frac \pi 2)$ to $(0,+\infty)$.
Next, consider the function
\begin{equation}
	\theta\mapsto \sqrt{\mu_\Omega(f^{-1}(\theta)) + m^2}.
\end{equation}
From the previous comments and \Cref{thm:PropertiesMu}~$(ii)$, this function is continuous, strictly decreasing, and bijective from $(-\frac \pi 2,\frac \pi 2)$ to 
$\big(m,\sqrt{\Lambda_\Omega+m^2}\big)$. 
Our goal now is to show that this function actually is $\theta \mapsto \lambda_\Omega(\theta)$, which would conclude the proof of the result. 
To check it, simply note that if $\theta\in(-\frac \pi 2,\frac \pi 2)$ then \Cref{lm:1st_to_1st} shows that 
$T(\theta, \lambda_\Omega(\theta))=(a, \mu_\Omega(a))$ for some $a>0$. Now, applying $T^{-1}$ to this identity and looking at the components of $T^{-1}$ we realize that $\theta=f(a)$ and, thus, that
$$\sqrt{\mu_\Omega(f^{-1}(\theta)) + m^2}
=\sqrt{\mu_\Omega(a) + m^2}
=\lambda_\Omega(\theta),$$ as desired.
\end{proof}

\subsection{Proof of the equivalence of conjectures}

With all the previous ingredients in hand, we are now ready to prove \Cref{Th:eqiv_conj_point} which, in turn, will lead to
the equivalence of \Cref{Conj:Dirac} for the quantum dot Dirac operators and \Cref{Conj:RobinType} for the $\overline\partial$-Robin Laplacians, as stated in \Cref{l:eqiv_conj_glob}.

\begin{proof}[Proof of \Cref{Th:eqiv_conj_point}]
Let us address the proof of $(i)$. As a suggestion for the reader, it may be helpful to read the proof having \Cref{Fig:ProofThm1} in mind.

\begin{figure}[h]
	\begin{tikzpicture}[scale=2]
		\draw[->] (0,0) -- (2.5,0) node[anchor=north] {$a$};
		\draw[->] (0,0) -- (0,1.8);
		
		\draw[thick, domain=0:2.2, smooth, variable=\x, CBorange] 
		plot ({\x}, {2*\x/(1+\x)}) node[right] {$\mu_D$};
		
		\draw[thick, domain=0.5:0.85, smooth, variable=\x, CBorange] 
		plot ({\x}, {2*\x/(1+\x) + 2*1.4/(1+1.4) -2*0.7/(1+0.7) }) node[right] {$\mu_\Omega$};
		
		\draw[dashed] (0.7,0) -- (0.7,{2*1.4/(1+1.4)});
		\draw[dashed] (1.4,0) -- (1.4,{2*1.4/(1+1.4)});
		\draw[dashed] (0,{2*1.4/(1+1.4)}) -- (1.4,{2*1.4/(1+1.4)});
		
		\fill (0.7,0) circle (0.02);
		\node[below] at (0.7,0) {$a$};
		\fill (1.4,0) circle (0.02);
		\node[below] at (1.4,0) {$a^*$};
		\node[left] at (0,{2*1.4/(1+1.4) -0.3}) 
		{$\begin{aligned}
				\mu_\Omega(a) \\
				\rotatebox{90}{$=$} \quad \\[-6pt]
				\!\!\! \!\!\! \lambda_D(\theta)^2 - m^2
			\end{aligned}$};
		
		\begin{scope}[xshift=5.1cm] 
			
			\draw[->] (-1.57,0) -- (1.8,0) node[anchor=north] {$\theta$};
			\draw[->] (-1.57,0) -- (-1.57,1.8);
			
			\draw[thick, domain=-1.57:1.57, smooth, variable=\x, CBgreen]
			plot ({\x}, {1-0.5*\x/(2-\x/2)}) node[right] {$\lambda_D$};
			
			\draw[thick, domain=0.6:1, smooth, variable=\x, CBgreen]
			plot ({\x}, {1-0.5*\x/(2-\x/2) -0.5*(-1)/(2-(-1)/2)  +0.5*0.8/(2-0.8/2)}) node[right] {$\lambda_\Omega$};
			
			\draw[dashed] (-1,0) -- (-1,{1-0.5*(-1)/(2-(-1)/2)});
			\draw[dashed] (0.8,0) -- (0.8,{1-0.5*(-1)/(2-(-1)/2)}); 
			
			\draw[dashed] (-1.57,{1-0.5*(-1)/(2-(-1)/2)}) -- (0.8,{1-0.5*(-1)/(2-(-1)/2)});
			
			\fill (-1,0) circle (0.02);
			\node[below] at (-1,0) {$\theta$};
			\fill (0.8,0) circle (0.02);
			\node[below] at (0.8,0) {$\theta^*$};
			\node[left] at (-1.57,{1-0.5*(-1)/(2-(-1)/2)}) {$\lambda_\Omega(\theta^*)$};
			
		\end{scope}
	\end{tikzpicture}
	\caption{Schematic representation of the proof of $(i)$.}
	\label{Fig:ProofThm1}
\end{figure}
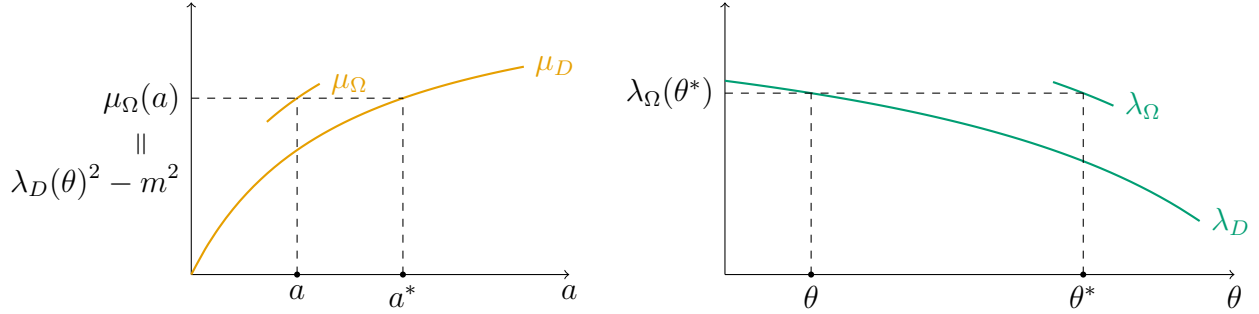

First of all, let us check that $a$ given in~$(i)$, that is,
\begin{equation}\label{def:a_as_theta_conjProof}
	a:=\mu_\Omega^{-1}(\lambda_D(\theta)^2-m^2),
\end{equation} 
is a well defined positive number for all $\theta\in(-\frac \pi 2, \frac \pi 2)$. Thanks to \Cref{thm:PropertiesLambda} and the Faber-Krahn inequality for the Dirichlet Laplacian \cite{Faber1923,Krahn1925}, we see that 
\begin{equation}
0<\lambda_D(\theta)^2-m^2<\Lambda_D\leq\Lambda_\Omega
\quad\text{for all $\theta\in{\textstyle(-\frac \pi 2, \frac \pi 2)}$.}
\end{equation}
Since the function $a\mapsto \mu_\Omega(a)$ is bijective from $(0,+\infty)$ to $(0,\Lambda_\Omega)$ by \Cref{thm:PropertiesMu}~$(ii)$, given $\theta\in(-\frac \pi 2, \frac \pi 2)$ there exists a unique $a>0$ such that 
$\mu_\Omega(a)=\lambda_D(\theta)^2-m^2$, as desired.

Next, given $\theta\in(-\frac \pi 2, \frac \pi 2)$, let $a$ be as in \eqref{def:a_as_theta_conjProof} and assume that 
$\mu_\Omega(a)> \mu_D(a)$.
Set 
$(a^*,\mu^*):=T(\theta,\lambda_D(\theta))$, that is,
\begin{equation}\label{eq:mu_as_lambda_D}
a^*:=(\lambda_D(\theta)+m)\vartheta(\theta)
\quad\text{and}\quad
\mu^*:=\lambda_D(\theta)^2-m^2.
\end{equation}
From \Cref{lm:1st_to_1st} we actually see that 
$\mu^*=\mu_D(a^*)$. Therefore, using the definition of $\mu^*$ in \eqref{eq:mu_as_lambda_D}, \eqref{def:a_as_theta_conjProof}, and the assumption in $(i)$, we get that
$\mu_D(a^*)=\mu^*=\lambda_D(\theta)^2-m^2
=\mu_\Omega(a)>\mu_D(a)$. Now, since the function $\mu_D$ is strictly increasing by \Cref{thm:PropertiesMu}~$(ii)$, we find that 
\begin{equation}\label{eq:a*>a}
a^*> a.
\end{equation}
Now, set $(\theta^*,\lambda^*):=T^{-1}(a,\mu_\Omega(a))$, that is,
\begin{equation}\label{eq:lambda_as_mu_omega}
\theta^*:=\vartheta^{-1}\bigg({\frac{a}{\sqrt{\mu_\Omega(a)+m^2}+m}}\bigg)
\quad\text{and}\quad
\lambda^*:=\sqrt{\mu_\Omega(a) + m^2}.
\end{equation}
On the one hand, as before, from \Cref{lm:1st_to_1st} we  see that $\lambda^*=\lambda_\Omega(\theta^*)$. On the other hand, replacing \eqref{def:a_as_theta_conjProof} in \eqref{eq:lambda_as_mu_omega} we deduce that
\begin{equation}\label{eq:theta*_as_lambda_Omega}
\theta^*=\vartheta^{-1}\bigg({\frac{a}{\lambda_D(\theta)+m}}\bigg)
\quad\text{and}\quad
\lambda^*=\lambda_D(\theta).
\end{equation}
In particular, we obtain that 
\begin{equation}\label{eq:rel_theta*_theta}
\lambda_\Omega(\theta^*)=\lambda_D(\theta).
\end{equation}
In addition, using \eqref{eq:a*>a} and \eqref{eq:mu_as_lambda_D} in \eqref{eq:theta*_as_lambda_Omega}, we deduce that
\begin{equation}
\vartheta(\theta^*)=\frac{a}{\lambda_D(\theta)+m}
< \frac{a^*}{\lambda_D(\theta)+m}
=\vartheta(\theta).
\end{equation}
From this, and since the function $\vartheta$ is strictly decreasing, we get that 
$\theta^*> \theta$.
Finally, using that $\theta^*> \theta$, that the function 
$\lambda_\Omega$ is strictly decreasing, and  \eqref{eq:rel_theta*_theta}, we conclude that
$\lambda_\Omega(\theta)>
\lambda_\Omega(\theta^*)=\lambda_D(\theta)$. This finishes the proof of $(i)$.

Let us now address the proof of $(ii)$, which will be completely analogous to the one of $(i)$. As before, it may be helpful to have \Cref{Fig:ProofThm2} in mind while reading this proof.

\begin{figure}[h]
	\begin{tikzpicture}[scale=2]
		
		\draw[->] (-1.57,0) -- (1.8,0) node[anchor=north] {$\theta$};
		\draw[->] (-1.57,0) -- (-1.57,1.8);
		
		\draw[thick, domain=-1.57:1.57, smooth, variable=\x, CBgreen]
		plot ({\x}, {1-0.5*\x/(2-\x/2)}) node[right] {$\lambda_D$};
		
		\draw[thick, domain=0.6:1, smooth, variable=\x, CBgreen]
		plot ({\x}, {1-0.5*\x/(2-\x/2) -0.5*(-1)/(2-(-1)/2)  +0.5*0.8/(2-0.8/2)}) node[right] {$\lambda_\Omega$};
		
		\draw[dashed] (-1,0) -- (-1,{1-0.5*(-1)/(2-(-1)/2)});
		\draw[dashed] (0.8,0) -- (0.8,{1-0.5*(-1)/(2-(-1)/2)}); 
		
		\draw[dashed] (-1.57,{1-0.5*(-1)/(2-(-1)/2)}) -- (0.8,{1-0.5*(-1)/(2-(-1)/2)});
		
		\fill (-1,0) circle (0.02);
		\node[below] at (-1,0) {$\theta^*$};
		\fill (0.8,0) circle (0.02);
		\node[below] at (0.8,0) {$\theta$};
		\node[left] at (-1.57,{1-0.5*(-1)/(2-(-1)/2)}) {$\lambda_\Omega(\theta)$};

		\begin{scope}[xshift=3.3cm] 
			
			\draw[->] (0,0) -- (2.5,0) node[anchor=north] {$a$};
			\draw[->] (0,0) -- (0,1.8);
			
			\draw[thick, domain=0:2.2, smooth, variable=\x, CBorange] 
			plot ({\x}, {2*\x/(1+\x)}) node[right] {$\mu_D$};
			
			\draw[thick, domain=0.5:0.85, smooth, variable=\x, CBorange] 
			plot ({\x}, {2*\x/(1+\x) + 2*1.4/(1+1.4) -2*0.7/(1+0.7) }) node[right] {$\mu_\Omega$};
			
			\draw[dashed] (0.7,0) -- (0.7,{2*1.4/(1+1.4)});
			\draw[dashed] (1.4,0) -- (1.4,{2*1.4/(1+1.4)});
			\draw[dashed] (0,{2*1.4/(1+1.4)}) -- (1.4,{2*1.4/(1+1.4)});
			
			\fill (0.7,0) circle (0.02);
			\node[below] at (0.7,0) {$a^*$};
			\fill (1.4,0) circle (0.02);
			\node[below] at (1.4,0) {$a$};
			\node[left] at (0,{2*1.4/(1+1.4)}) {$\mu_\Omega(a^*)$};

		\end{scope}
	\end{tikzpicture}
	\caption{Schematic representation of the proof of $(ii)$.  }
	\label{Fig:ProofThm2}
\end{figure}
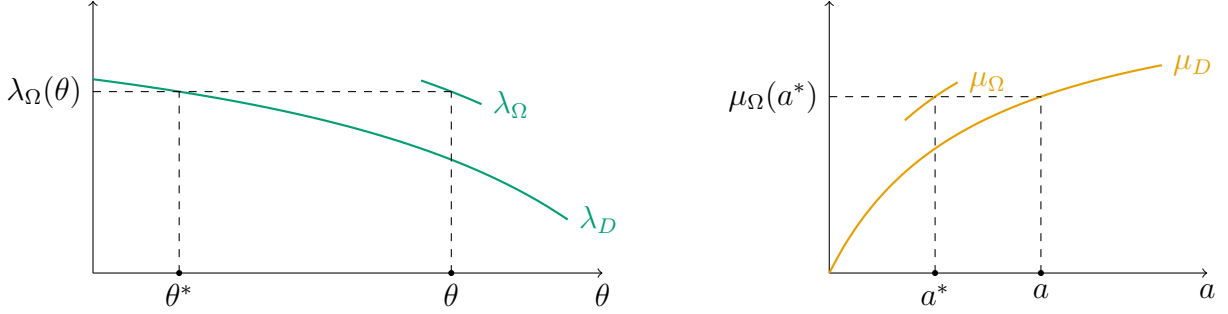

First of all, let us check that the value $\theta$ given in $(ii)$, that is,
\begin{equation}\label{def:theta_as_a_conjProof}
	\theta:= \lambda_\Omega^{-1}\big({\textstyle\sqrt{\mu_D(a)+m^2}}\big),
\end{equation}
is a well defined real number in $(-\frac \pi 2, \frac \pi 2)$ for all $a>0$. Thanks to \Cref{thm:PropertiesMu}~$(ii)$ and the Faber-Krahn inequality for the Dirichlet Laplacian, we see that 
\begin{equation}
m^2<\mu_D(a)+m^2
<\Lambda_D+m^2\leq\Lambda_\Omega+m^2
\quad\text{for all $a>0$.}
\end{equation}
Since the function $\theta\mapsto \lambda_\Omega(\theta)$ is bijective from $(-\frac \pi 2, \frac \pi 2)$ to 
$\big(m,\sqrt{\Lambda_\Omega+m^2}\big)$ by \Cref{thm:PropertiesLambda}, given $a>0$  there exists a unique $\theta\in(-\frac \pi 2, \frac \pi 2)$ such that $\lambda_\Omega(\theta)=\sqrt{\mu_D(a)+m^2}$, as desired.

Next, given $a>0$, let $\theta$ be as in \eqref{def:theta_as_a_conjProof} and assume that 
$\lambda_\Omega(\theta)> \lambda_D(\theta)$. Set 
$(\theta^*,\lambda^*):=T^{-1}(a,\mu_D(a))$, that is,
\begin{equation}\label{eq:lambda_as_mu_D}
\theta^*:=
\vartheta^{-1}\Big({\textstyle\frac{a}{\sqrt{\mu_D(a)+m^2}+m}}\Big)
\quad\text{and}\quad
\lambda^*:=\sqrt{\mu_D(a) + m^2}.
\end{equation}
From \Cref{lm:1st_to_1st} we actually see that 
$\lambda^*=\lambda_D(\theta^*)$. Therefore, using the definition of $\lambda^*$ in \eqref{eq:lambda_as_mu_D}, \eqref{def:theta_as_a_conjProof}, and the assumption in $(ii)$, we get that
$\lambda_D(\theta^*)=\lambda^*=\sqrt{\mu_D(a)^2+m^2}
=\lambda_\Omega(\theta)>\lambda_D(\theta)$. Now, since the function $\lambda_D$ is strictly decreasing by \Cref{thm:PropertiesLambda}, we find that 
\begin{equation}\label{eq:theta*>theta}
\theta^*< \theta.
\end{equation}
Now, set $(a^*,\mu^*):=T(\theta,\lambda_\Omega(\theta))$, that is,
\begin{equation}\label{eq:mu_as_lambda_omega}
a^*:=(\lambda_\Omega(\theta)+m)\vartheta(\theta)
\quad\text{and}\quad
\mu^*:=\lambda_\Omega(\theta)^2-m^2.
\end{equation}
On the one hand, as before, from \Cref{lm:1st_to_1st} we  see that $\mu^*=\mu_\Omega(a^*)$. On the other hand, applying \eqref{def:theta_as_a_conjProof} to \eqref{eq:mu_as_lambda_omega} we deduce that
\begin{equation}\label{eq:a*_as_mu_Omega}
a^*=\big({\textstyle\sqrt{\mu_D(a)+m^2}}+m\big)\vartheta(\theta)
\quad\text{and}\quad
\mu^*=\mu_D(a).
\end{equation}
In particular, we obtain that 
\begin{equation}\label{eq:rel_a*_a}
\mu_\Omega(a^*)=\mu_D(a).
\end{equation}
In addition, thanks to \eqref{eq:theta*>theta} and the fact that the function $\vartheta$ is strictly decreasing, we see that 
$\vartheta(\theta)<\vartheta(\theta^*)$. Using this and
\eqref{eq:lambda_as_mu_D} in \eqref{eq:a*_as_mu_Omega}, we deduce that
\begin{equation}
a^*=\big({\textstyle\sqrt{\mu_D(a)+m^2}}+m\big)
\vartheta(\theta)
<\big({\textstyle\sqrt{\mu_D(a)+m^2}}+m\big)
\vartheta(\theta^*)
=a. 
\end{equation}
Finally, using that $a^*< a$, that the function 
$\mu_\Omega$ is strictly increasing, and  \eqref{eq:rel_a*_a}, we conclude that
$\mu_\Omega(a)>
\mu_\Omega(a^*)=\mu_D(a)$. This finishes the proof of $(ii)$.
\end{proof}

\begin{proof}[Proof of \Cref{l:eqiv_conj_glob}]
If \Cref{Conj:RobinType} holds true then \Cref{Conj:Dirac} holds true by \Cref{Th:eqiv_conj_point}~$(i)$. The reverse implication follows analogously but using now \Cref{Th:eqiv_conj_point}~$(ii)$.
\end{proof}

\section{Asymptotic regimes} \label{sec:asympt_regimes}

This section is devoted to the proofs of \Cref{thm:VerifyConjLimits,Th:DiscOptimalThetaAsympt}, and of \Cref{prop:muPrimeOrigin}. 
Recall that \Cref{thm:VerifyConjLimits} refers to the optimality of the disk for the first eigenvalue of $\RR_a$ in the asymptotic regimes 
$a\downarrow0$ and $a\uparrow+\infty$, and \Cref{Th:DiscOptimalThetaAsympt} refers to the optimality of the disk for the first nonnegative eigenvalue of $\D_\theta$ in the asymptotic regimes
$\theta\uparrow\frac{\pi}{2}$ and $\theta\downarrow-\frac{\pi}{2}$. 
As we mentioned in the introduction, the proofs in the cases $a\uparrow+\infty$ and $\theta\downarrow-\frac{\pi}{2}$ will essentially follow by the Faber-Krahn inequality for the Dirichlet Laplacian. 
Instead, the case 
$\theta\uparrow\frac{\pi}{2}$ will be derived from \Cref{Th:eqiv_conj_point}~$(i)$ and the case $a\downarrow 0$, the latter being the most complicated to prove.

As an advance to help the reader, next we recall the main ideas used in the case $a\downarrow 0$.
Since 
$\lim_{a\downarrow 0}\mu_\Omega(a)=0=\lim_{a\downarrow 0}\mu_D(a)$, in order to prove that if $\Omega$ is not a disk then
$\mu_\Omega(a)>\mu_D(a)$ for all $a>0$ small enough, it suffices to show that the slope of the function 
$\mu_\Omega$ when departing from the origin is strictly bigger than the one of $\mu_D$. 
This last claim will follow from \Cref{prop:muPrimeOrigin}. 
More precisely, we will first give a variational characterization of the slope as
\begin{equation}
	\label{eq:slope_as_inf_quotientProof}
	\lim_{a\downarrow 0} \frac{\mu_\Omega(a)}{a}	= S_\Omega := \inf_{u\in E(\Omega)\setminus\{0\}:\,
		\partial_{\bar z}u=0\text{ in }\Omega}
	\frac{\int_{\partial\Omega}|u|^2}{\int_{\Omega}|u|^2}.
\end{equation}
This is \eqref{eq:slope_as_inf_quotient1} in \Cref{prop:muPrimeOrigin}~$(ii)$.
Then, we will use this characterization to prove, for simply connected domains, a sharp lower bound of the slope for which the disks are the only minimizers; this will yield $S_\Omega \geq 2 \sqrt{\pi/ |\Omega|}$, which is \eqref{eq:slope_as_inf_quotient2} in \Cref{prop:muPrimeOrigin}.
This last lower bound is precisely the content of the next result, which can be rephrased as the optimal constant for an embedding of a Hardy space into a Bergman space; see \Cref{sec:HardyBergman} for more details.

\begin{proposition} \label{prop:CarlemanL2}
	Let $\Omega\subset\R^2$ be a simply connected bounded domain with $C^2$ boundary. Then,
	\begin{equation}\label{eq:Carleman_slope}
		\|u\|_{L^2(\Omega)}^2 \leq \frac{1}{2} \sqrt{\frac{|\Omega|}{\pi}}\|u\|_{L^2(\partial\Omega)}^2
	\end{equation}
	for all $u\in E(\Omega)$ such that $\partial_{\bar z}u=0$ in $\Omega$. Moreover, if $u$ does not vanish identically, then the equality in \eqref{eq:Carleman_slope} holds if and only if $\Omega$ is a disk and $u$ is a constant function.
\end{proposition}
Our proof of this bound requires the Riemann mapping theorem and a quantitative version of a theorem for holomorphic functions in the unit disk due to Hardy and Littlewood; see \Cref{lemma:Carleman}.
Since the arguments require some technical tools from complex analysis, we postpone the proof until \Cref{sec:HardyBergman}.
Here, assuming \Cref{prop:CarlemanL2} proved, we establish  \Cref{prop:muPrimeOrigin}.

\begin{proof}[Proof of \Cref{prop:muPrimeOrigin}]
We begin the proof of the theorem addressing $(ii)$, that is, \eqref{eq:slope_as_inf_quotientProof}. 
As we will see, the arguments used for its proof will also yield the first part of $(i)$. 

As a preliminary comment, note that 
$\lim_{a\downarrow 0} \mu_\Omega(a)/a$ exists and is a positive real number thanks to \Cref{lm:a/mu_decrease}, since the function $a\mapsto\mu_\Omega(a)/a$ is positive, strictly decreasing in $(0,+\infty)$, and bounded from above. 
The easy step to get $(ii)$ is to prove that $\lim_{a\downarrow 0} \mu_\Omega(a)/a \leq S_\Omega$.
To check this, simply restrict the infimum in \eqref{eq:RQ_Rodzin_mu} to the set $\{u\in E(\Omega)\setminus\{0\}:\,\partial_{\bar z}u=0\text{ in }\Omega\}$ and then divide by $a$. This actually shows that  
${\mu_\Omega(a)}/{a}\leq S_\Omega$ for all $a>0$, and not only in the limit $a\downarrow 0$. 

Next, we will prove that $\lim_{a\downarrow 0} \mu_\Omega(a)/a \geq S_\Omega$. 
This step will follow by compactness. For every $a>0$, let $u_\Omega(a)$ be as in \Cref{thm:PropertiesMu}~$(i)$. By normalization, we can assume that $\|u_\Omega(a)\|_{L^2(\Omega)}=1$ for all $a>0$. This, combined with \eqref{ineq:est_mu/a_per_vol},  leads to
\begin{equation}\label{eq_est:mu_u_omega(a)}
0\leq 4 \int_\Omega |\partial_{\bar z} u_\Omega(a)|^2  
+ a \int_{\partial\Omega} |u_\Omega(a)|^2
= \mu_\Omega(a)
\leq a\frac{|\partial\Omega|}{|\Omega|}\quad
\text{for all $a>0$}.
\end{equation}
Therefore, $\lim_{a\downarrow 0}\|\partial_{\bar z} u_\Omega(a)\|_{L^2(\Omega)}=0$ and  
$\|u_\Omega(a)\|_{L^2(\partial\Omega)}
\leq \sqrt{|\partial\Omega|/|\Omega|}$ for all $a>0$. In particular,
\begin{equation}\label{ineq:compacness_a_to_0}
\|u_\Omega(a)\|_{E(\Omega)} := \int_\Omega|u_\Omega(a)|^2 + \int_\Omega|\partial_{\bar z} u_\Omega(a)|^2 + \int_{\partial\Omega} |u_\Omega(a)|^2 \leq C
\quad\text{for all $a\in(0,1]$},
\end{equation}
for some $C>0$ independent of $a$.

On the one hand, since $E(\Omega)$ is compactly embedded in $L^2(\Omega)$ by \cite[Lemma~3.1]{Duran2025},  
form \eqref{ineq:compacness_a_to_0} we deduce that there exist $u_\Omega\in L^2(\Omega)$ and a sequence $a_k\downarrow 0$ as $k\uparrow+\infty$ such that 
\begin{equation}\label{eq:lim_compactness_strong}
\begin{split}
&\lim_{k\uparrow+\infty}\|u_\Omega-u_\Omega(a_k)\|_{L^2(\Omega)} = 0 \quad\text{and, consequently, that}\\
&\partial_{\bar z}u_\Omega=\lim_{k\uparrow+\infty}\partial_{\bar z}u_\Omega(a_k)\quad
\text{in the sense of distributions in $\Omega$.}
\end{split}
\end{equation}
The first statement in \eqref{eq:lim_compactness_strong} yields
$\|u_\Omega\|_{L^2(\Omega)}=1$. From the second statement in \eqref{eq:lim_compactness_strong} and the fact that 
$\lim_{a\downarrow 0}\|\partial_{\bar z} u_\Omega(a)\|_{L^2(\Omega)}=0$, we deduce, firstly, that $\partial_{\bar z}u_\Omega=0$ in the sense of distributions in $\Omega$ ---which entails 
$\partial_{\bar z}u_\Omega\in L^2(\Omega)$--- 
and, secondly, that
\begin{equation}
\lim_{k\uparrow+\infty}\|\partial_{\bar z}u_\Omega-\partial_{\bar z}u_\Omega(a_k)\|_{L^2(\Omega)}=0.
\end{equation}
This last conclusion, combined with the first statement in \eqref{eq:lim_compactness_strong} and \cite[Lemma~2.3]{Benguria2017Self} show that $\lim_{k\uparrow+\infty}\|u_\Omega(a_k)-u_\Omega\|_{H^{-1/2}(\partial\Omega)}=0$. In  particular,
\begin{equation}\label{eq:lim_compact_1}
\lim_{k\uparrow+\infty}\langle u_\Omega-u_\Omega(a_k), v \rangle_{H^{-1/2}(\partial\Omega),H^{1/2}(\partial\Omega)}=0\quad\text{for all $v\in H^{1/2}(\partial\Omega)$.}
\end{equation}

On the other hand, since $\|u_\Omega(a)\|_{L^2(\partial\Omega)}\leq\sqrt{|\partial\Omega|/|\Omega|}$ for all $a>0$, combining weak-$*$ compactness (Banach–Alaoglu theorem) and Riesz-Fr\'echet theorem on the Hilbert space $L^2(\partial\Omega)$, we deduce that there exist
$u_\Omega^*\in L^2(\partial\Omega)$ and a subsequence of $\{a_{k}\}_k$, which for simplicity we call again $\{a_{k}\}_k$, such that 
\begin{equation}\label{eq:lim_compact_3}
\lim_{k\uparrow+\infty}\langle v,u_\Omega^*-u_\Omega(a_k) \rangle_{L^2(\partial\Omega)}=0\quad\text{for all $v\in L^2(\partial\Omega)$.}
\end{equation}
Since $H^{1/2}(\partial\Omega)\subset L^2(\partial\Omega)$, and in view of our convention for the pairing \eqref{eq:dual_pairing_L2}, \eqref{eq:lim_compact_3} gives
\begin{equation}\label{eq:lim_compact_2}
\begin{split}
\lim_{k\uparrow+\infty}\langle u_\Omega^*-u_\Omega(a_k), v \rangle_{H^{-1/2}(\partial\Omega),H^{1/2}(\partial\Omega)}
&=\lim_{k\uparrow+\infty}\langle v,u_\Omega^*-u_\Omega(a_k)  \rangle_{L^2(\partial\Omega)}
=0
\end{split}
\end{equation}
for all $v\in H^{1/2}(\partial\Omega)$.
Using this and \eqref{eq:lim_compact_1}, we conclude that 
$u_\Omega=u_\Omega^*$ in $H^{-1/2}(\partial\Omega)$.  To be more precise, as a functional in $H^{-1/2}(\partial\Omega)$, $u_\Omega$ is equal to
$\langle\cdot,u_\Omega^*\rangle_{L^2(\partial\Omega)}$ for some $u_\Omega^*\in L^2(\partial\Omega)$. This is what, abusing notation, we are denoting by $u_\Omega\in L^2(\partial\Omega)$. 

Summing up, from the previous arguments we have seen that $\|u_\Omega\|_{L^2(\Omega)}=1$, that $\partial_{\bar z}u_\Omega=0$ in~$\Omega$, that $u_\Omega\in L^2(\partial\Omega)$ ---these three facts yield $u_\Omega\in E(\Omega)\setminus\{0\}$---, and that $u_\Omega=\lim_{k\uparrow+\infty}u_\Omega(a_k)$ weakly in $L^2(\partial\Omega)$ ---recall \eqref{eq:lim_compact_3}. In particular, 
\begin{equation}
\begin{split}
\liminf_{k\uparrow+\infty}\|u_\Omega(a_{k})
\|_{L^2(\partial\Omega)}&=
\liminf_{k\uparrow+\infty}\sup_{w\in L^2(\partial\Omega):\,\|w\|_{L^2(\partial\Omega)}=1}\int_{\partial\Omega} u_\Omega(a_{k})\overline w\\
&\geq\liminf_{k\uparrow+\infty}\int_{\partial\Omega} u_\Omega(a_{k})\frac{\overline{u_\Omega}}{\|u_\Omega\|_{L^2(\partial\Omega)}}=\|u_\Omega\|_{L^2(\partial\Omega)};
\end{split}
\end{equation}
this is nothing but the weakly lower semicontinuity of 
$\|\cdot\|_{L^2(\partial\Omega)}$.
With all these ingredients in hand, the fact that $\lim_{a\downarrow 0} \mu_\Omega(a)/a \geq S_\Omega$ follows easily. 
Indeed,
\begin{equation}
\begin{split}
\lim_{a\downarrow 0}\frac{\mu_\Omega(a)}{a}
&=\lim_{k\uparrow+\infty}\frac{\mu_\Omega(a_k)}{a_k}
=\lim_{k\uparrow+\infty}
\Big(\frac{4}{a_k}\int_\Omega |\partial_{\bar z} u_\Omega(a_k)|^2  
+\int_{\partial\Omega} |u_\Omega(a_k)|^2\Big)\\
&\geq\liminf_{k\uparrow+\infty}\|u_\Omega(a_{k})\|_{L^2(\partial\Omega)}^2
\geq\|u_\Omega\|_{L^2(\partial\Omega)}^2
=\frac{\|u_\Omega\|_{L^2(\partial\Omega)}^2}
{\|u_\Omega\|_{L^2(\Omega)}^2}\geq S_\Omega,
\end{split}
\end{equation}
as desired, completing the proof of \eqref{eq:slope_as_inf_quotientProof}. 
Actually, looking at this chain of inequalities (which, as we have proved, are all equalities), we deduce that the infimum in the definition of $S_\Omega$ ---see~\eqref{def:S_Omega}--- is attained by $u_\Omega$, which proves the first part of $(i)$.

We next show that the minimizer $u_\Omega \in E(\Omega)$ that we have found actually belongs to $H^1(\Omega)$. 
For this, we will use crucially that $u_\Omega$ is the limit of functions $u_\Omega(a_k)\in \Dom(\RR_{a_k})$.
On the one hand, we have seen that there exists a sequence $a_k\downarrow 0$ as $k\uparrow+\infty$ and $u_\Omega(a_k)$ as in \Cref{thm:PropertiesMu}~$(i)$ such that $u_\Omega(a_k) \to u_\Omega$ and $\partial_{\bar z} u_\Omega(a_k) \to 0$ strongly in $L^2(\Omega)$ as $k\uparrow+\infty$. On the other hand, since $u_\Omega(a_k) \in \mathrm{Dom}(\RR_{a_k})$, by \cite[Lemma~3.4]{Duran2025} there holds
\begin{equation}
    \begin{split}
        \|u_\Omega(a_k)\|_{H^1(\Omega)}^2 & \leq C_\Omega \Big( \|u_\Omega(a_k)\|_{L^2(\Omega)}^2 + \|\partial_{\bar z} u_\Omega(a_k)\|_{L^2(\Omega)}^2 + \frac{1}{a_k^2} \|\Delta u_\Omega(a_k)\|_{L^2(\Omega)}^2 \Big) \\
        & = C_\Omega \Big( \Big(1+\dfrac{\mu_\Omega(a_k)^2}{a_k^2} \Big) \|u_\Omega(a_k)\|_{L^2(\Omega)}^2 + \|\partial_{\bar z} u_\Omega(a_k)\|_{L^2(\Omega)}^2 \Big)
    \end{split}
\end{equation}
for some constant $C_\Omega>0$ depending only on $\Omega$, where in the last equality we have used that $-\Delta u_\Omega(a_k) = \mu_\Omega(a_k) u_\Omega(a_k)$ in $L^2(\Omega)$.  This estimate, together with \eqref{eq:slope_as_inf_quotientProof} and the aforementioned convergence of $u_\Omega(a_k)$ to $u_\Omega$, leads to the boundedness of $\|u_\Omega(a_k)\|_{H^1(\Omega)}$ uniformly in $k$. Combining the compact embedding of $H^1(\Omega)$ in $L^2(\Omega)$ and the weak-$*$ compactness of $H^1(\Omega)$, we deduce that there exists $u_\star \in H^1(\Omega)$ and a subsequence of $\{a_k\}_k$, which we call again $\{a_k\}_k$, such that $u_\Omega(a_k) \to u_\star$ in $L^2(\Omega)$ as $k\uparrow+\infty$. Since we already had that $u_\Omega(a_k) \to u_\Omega$ in $L^2(\Omega)$, this yields that $u_\Omega=u_\star\in H^1(\Omega)$, as desired.
   
To conclude the proof of $(i)$, it only remains to show \eqref{eq:EL_u_Omega}, which is nothing but the Euler-Lagrange equation for any minimizer $u$ of \eqref{def:S_Omega}. Given $v\in E(\Omega)$ with
$\partial_{\bar z}v=0$ in $\Omega$, set
\begin{equation}
f(t):=\frac{\int_{\partial\Omega}|u+tv|^2}{\int_{\Omega}|u+tv|^2}
=\frac{\int_{\partial\Omega}|u|^2
+2t\re\big(\int_{\partial\Omega}u\,\overline v\big)
+t^2\int_{\partial\Omega}|v|^2}
{\int_{\Omega}|u|^2
+2t\re\big(\int_{\Omega}u\,\overline v\big)
+t^2\int_{\Omega}|v|^2}
\end{equation}
for all $t\in\R$ with $|t|$ small enough.
If $u$ is a minimizer of \eqref{def:S_Omega}, we deduce that 
$f(0)\leq f(t)$ for all $|t|$ small enough and, thus,
\begin{equation}
0=\frac{d}{dt}f(0)=2\frac{\re\big(\int_{\partial\Omega}u\,\overline v\big)\int_{\Omega}|u|^2-\int_{\partial\Omega}|u|^2\re\big(\int_{\Omega}u\,\overline v\big)}
{\big(\int_{\Omega}|u|^2\big)^2},
\end{equation}
which yields 
\begin{equation}\label{eq:EL_u_Omega_real}
\re\Big(\int_{\partial\Omega}u\,\overline v\Big)
=S_\Omega\re\Big(\int_{\Omega}u\,\overline v\Big)
\quad\text{for all $v\in E(\Omega)$ with
$\partial_{\bar z}v=0$ in $\Omega$.}
\end{equation}
This proves \eqref{eq:EL_u_Omega} at the level or real parts. To get the equality also for the imaginary parts, given $v\in E(\Omega)$ with $\partial_{\bar z}v=0$ in $\Omega$, simply apply \eqref{eq:EL_u_Omega_real} to the function 
$-iv$.

Finally, note that \eqref{eq:slope_as_inf_quotient2}, that is, $S_\Omega \geq 2 \sqrt{\pi/ |\Omega|}$, is a consequence of \eqref{eq:slope_as_inf_quotientProof} and \Cref{prop:CarlemanL2}.
Thus, it only remains to prove that the equality in \eqref{eq:slope_as_inf_quotient2} holds if and only if $\Omega$ is a disk. 
On the one hand, if $\Omega$ is a disk of radius $r>0$ then, using \eqref{eq:slope_as_inf_quotient2}, $\lim_{a\downarrow 0} \mu_\Omega(a)/a = S_\Omega$, and comparing the quotient in the definition of  $S_\Omega$ ---see \eqref{def:S_Omega}--- with its the value for constant functions, we deduce that
\begin{equation}
2 \sqrt{\frac{\pi}{|\Omega|}}
\leq\lim_{a\downarrow 0} \frac{\mu_\Omega(a)}{a}
=S_\Omega
\leq\frac{\int_{\partial\Omega}|1|^2}{\int_{\Omega}|1|^2}
=\frac{|\partial\Omega|}{|\Omega|}
=\frac{2}{r}
=2 \sqrt{\frac{\pi}{|\Omega|}},
\end{equation}
which gives the equality in \eqref{eq:slope_as_inf_quotient2}.
On the other hand, if the equality in \eqref{eq:slope_as_inf_quotient2} holds, by points~$(i)$ and $(ii)$ already proved, we deduce that
\begin{equation}
2 \sqrt{\frac{\pi}{|\Omega|}}
=\lim_{a\downarrow 0} \frac{\mu_\Omega(a)}{a}
=S_\Omega
=\frac{\int_{\partial\Omega}|u_\Omega|^2}{\int_{\Omega}|u_\Omega|^2}
\quad\text{for some $u_\Omega\in E(\Omega)\setminus\{0\}$ such 
$\partial_{\bar z}u_\Omega=0$ in $\Omega$}.
\end{equation}
From this and \Cref{prop:CarlemanL2} we conclude that $\Omega$ must be a disk (and also that $u_\Omega$ must be constant), as desired.
\end{proof} 

At this point, we are ready to prove \Cref{thm:VerifyConjLimits,Th:DiscOptimalThetaAsympt}. Namely, we show the asymptotic optimality of the disk, both for the first eigenvalue of the $\overline\partial$-Robin Laplacian and for the first nonnegative eigenvalue of the quantum dot Dirac operator.

\begin{proof}[Proof of \Cref{thm:VerifyConjLimits}]
Let $\Omega\subset \R^2$ be a bounded domain with $C^2$ boundary, assume that $\Omega$ is not a disk, and let $D\subset\R^2$ be a disk with the same area as $\Omega$. 
Then, $\Lambda_\Omega>\Lambda_D$ by \cite[Theorem~1.2]{DanersKennedy2007}. This, together with
\Cref{thm:PropertiesMu}~$(ii)$, leads to 
\begin{equation}
\lim_{a\to+\infty}\mu_\Omega(a)
=\Lambda_\Omega
>\Lambda_D=\lim_{a\to+\infty}\mu_D(a).
\end{equation}
Therefore,
$\mu_\Omega(a)>\mu_D(a)$
for all $a>0$ big enough, as desired.

Next, assume in addition that $\Omega$ is simply connected. Since $\Omega$ is not a disk, \Cref{prop:muPrimeOrigin} yields
\begin{equation}
\lim_{a\downarrow 0} \frac{\mu_\Omega(a)}{a}
> 2 \sqrt{\frac{\pi}{|\Omega|}}
=2 \sqrt{\frac{\pi}{|D|}}
=\lim_{a\downarrow 0} \frac{\mu_D(a)}{a}.
\end{equation}
From this, it follows that $\mu_\Omega(a)>\mu_D(a)$
for all $a>0$ small enough, as desired.
\end{proof}

\begin{proof}[Proof of \Cref{Th:DiscOptimalThetaAsympt}]
Let $\Omega\subset \R^2$ be a bounded domain with $C^2$ boundary, assume that $\Omega$ is not a disk, and let $D\subset\R^2$ be a disk with the same area as $\Omega$. 
Then, as in the previous proof, $\Lambda_\Omega>\Lambda_D$, which together with \Cref{thm:PropertiesLambda} leads to 
\begin{equation}
\lim_{\theta\to-\frac{\pi}{2}^{+}}\lambda_\Omega(\theta)
=\sqrt{\Lambda_\Omega+m^2}
>\sqrt{\Lambda_D+m^2}=\lim_{\theta\to-\frac{\pi}{2}^{+}}\lambda_D(\theta).
\end{equation}
Therefore, there exists $\theta_0\in(-\frac \pi 2,\frac \pi 2)$  such that  $\lambda_\Omega(\theta)>\lambda_D(\theta)$
for all $\theta\in(-\frac{\pi}{2},\theta_0)$, as desired.

From now on, assume in addition that $\Omega$ is simply connected. Then, by \Cref{thm:VerifyConjLimits}, there exists $a_1>0$ depending on $\Omega$ such that 
$\mu_\Omega(a)>\mu_D(a)$
for all $a\in(0,a_1)$. Now, combining \Cref{thm:PropertiesLambda} 
with \Cref{thm:PropertiesMu}~$(ii)$, identifying $y:=\lambda_D(\theta)^2-m^2$ we see that
\begin{equation}
\lim_{\theta\to\frac{\pi}{2}^{-}}
\mu_\Omega^{-1}(\lambda_D(\theta)^2-m^2)
=\lim_{y\downarrow 0}\mu_\Omega^{-1}(y)=0.
\end{equation}
This means that there exists 
$\theta_1\in(-\frac \pi 2,\frac \pi 2)$ 
(depending on $\Omega$) such that
$\mu_\Omega^{-1}(\lambda_D(\theta)^2-m^2)
\in(0,a_1)$ for all $\theta\in(\theta_1,\frac \pi 2)$. Combining the fact that 
$\mu_\Omega(a)>\mu_D(a)$
for all $a\in(0,a_1)$ with \Cref{Th:eqiv_conj_point}~$(i)$ ---in view of \eqref{def:a_as_theta_conj}---, we conclude that $\lambda_\Omega(\theta)>\lambda_D(\theta)$ for all $\theta\in(\theta_1,\frac{\pi}{2})$, as desired. 
\end{proof}

\section{Embedding of a Hardy space into a Bergman space} \label{sec:HardyBergman}

In this section we will establish \Cref{prop:CarlemanL2}, which was a crucial tool to establish our main results in the previous section.
We will begin the section recalling the quantitative version of a theorem of Hardy and Littlewood, next we will give a result relating the norm in a Hardy space with a certain boundary $L^2$-norm, and we will conclude by giving the proof of \Cref{prop:CarlemanL2}.

In 1932, Hardy and Littlewood proved in \cite[Theorem 31]{HardyLittlewood1932} the inclusion of Hardy spaces into Bergman spaces in the unit disk, but their proof did not give the sharp constant of the injection map. 
Based on the ideas of Carleman in \cite{Carleman1921} for his proof of the isoperimetric inequality using complex analytic methods, in 2003 Vukoti\'c gave, in \cite{Vukotic2003}, a rather elementary proof of the Hardy-Littlewood result which, moreover, yields the exact value of the norm of the injection map from Hardy spaces into Bergman spaces, as well as the extremal functions. 
This result will be the starting point of our developments in this section.
In order to state it in detail, let us first recall some terminology and basic facts on these spaces. 

For $\mathbb D:=\{z\in\C:\,|z|<1\}$, one says that $f:\mathbb D\to\C$ belongs to the \textbf{Hardy space} $\mathcal H^p(\mathbb D)$, for $p>0$, if $f$ is holomorphic in $\mathbb D$ and 
\begin{equation}\label{Hardy_p:condition}
	\sup_{0<r<1}M_p(r,f)<+\infty,\quad\text{where}\quad M_p(r, f) := \Big(\frac{1}{2\pi}\int_0^{2\pi} |f(r e^{i\phi})|^p\, d\phi\Big)^{1/p}.
\end{equation}
As observed by  Hardy in 1915 (see
\cite[Theorem 1.5]{Duren1970}), if $p>0$ and $f$ is holomorphic in $\mathbb D$ then the map 
$r\mapsto M_p(r, f)$ is nondecreasing, hence $\sup_{0<r<1}M_p(r,f)=\lim_{r\to 1^-}M_p(r,f)$. 
Moreover, if such limit is finite, it is known that $f(e^{i\phi})$ exists for almost every $\phi\in[0,2\pi)$ and 
\begin{equation}\label{Hardy_p:norm_1}
	\lim_{r\to1^-}\int_0^{2\pi}|f(r e^{i\phi})-f(e^{i\phi})|^p\, d\phi=0
	\quad\text{as well as}\quad
	\lim_{r\to1^-}f(r e^{i\phi})=f(e^{i\phi})
	\text{ for a.e. $\phi$;}
\end{equation}
see \cite[Theorems 2.2 and 2.6]{Duren1970}. Then, as described in \cite[page 23 in Section 2.3 for $p\geq1$ or, more generally, page 35 in Section 3.2 for $p>0$]{Duren1970}, 
one defines the norm
\begin{equation}\label{Hardy_p:norm_2}
	\|f\|_{\mathcal H^p(\mathbb D)} := \underset{r\to 1^-}{\lim } \, M_p(r,f)
	=\Big(\frac{1}{2\pi}\int_0^{2\pi} |f(e^{i\phi})|^p\, d\phi\Big)^{1/p}\quad\text{for $f\in \mathcal H^p(\mathbb D)$.}
\end{equation}
Finally, recall that the \textbf{Bergman space} $\mathcal A^p(\mathbb D)$, for $p>0$, is defined as the set of all functions $f$ holomorphic in $\mathbb D$ such that 
$$\|f\|_{\mathcal A^p(\mathbb D)}
:=\Big(\frac{1}{\pi}\int_{\mathbb D}|f|^p\Big)^{1/p}<+\infty.$$

After all these considerations, we are ready to state the  quantitative version of the theorem of Hardy and Littlewood given by Vukoti\'c. We must mention that a more general version of this result was previously shown by Burbea in \cite{Burbea1987}.\footnote{We thank J. Ortega-Cerd\`a for pointing out this reference, as well as for helpful discussions regarding Hardy and Bergman spaces.}

\begin{theorem}
	\label{lemma:Carleman}
	$($\cite[Theorem in page 534]{Vukotic2003}$)$
	For arbitrary $p>0$, every function $f$ in $\mathcal H^p(\mathbb D)$ belongs to $\mathcal A^{2p}(\mathbb D)$ and satisfies $\|f\|_{\mathcal A^{2p}(\mathbb D)}\leq\|f\|_{\mathcal H^p(\mathbb D)}$, with equality if and only if $f$ has the form
	\begin{equation}
		f(z) = c_1\Big(\frac{1}{1-c_2 z}\Big)^{2/p} \quad \text{for all $z\in{\mathbb D}$,}
	\end{equation}
	for some $c_1,c_2 \in \C$ with $|c_2| < 1$.
\end{theorem} 

We will only use \Cref{lemma:Carleman} for $p=2$. As can be seen from \cite{Vukotic2003}, this is a very simple case: its proof essentially follows from Taylor series expansions, orthogonality (that is, Parseval's identity), and the Cauchy-Schwarz inequality applied to the coefficients. We also mention that the general case $p>0$ in \Cref{lemma:Carleman} follows from the case $p=2$ by 
taking, for $f\in \mathcal H^p(\mathbb D)$, an analytic branch of $f^{p/2}$ after factoring out the zeroes of $f$ through a Blaschke product.

Now, given a simply connected bounded domain $\Omega\subset \R^2$ with $C^2$ boundary, the first step of our developments will be to bring \Cref{lemma:Carleman} for $p=2$ to the context of functions $u\in E(\Omega)$ such that $\partial_{\bar z}u=0$ in 
$\Omega$; recall \eqref{def:E} for the definition of $E(\Omega)$. This will be done using a conformal mapping from $\mathbb D$ to $\Omega$. However, since the boundary trace of a function in $E(\Omega)$ is defined in a Sobolev sense but the $\mathcal H^2(\mathbb D)$-norm is defined in terms of the integral means $M_2$ ---recall \eqref{Hardy_p:norm_2}---, as a preliminary step we find convenient to give a detailed proof of the fact that any holomorphic function $u\in E(\Omega)$ gives rise, through the conformal mapping, to a function $f\in\mathcal H^2(\mathbb D)$, whose norm $\|f\|_{\mathcal H^2(\mathbb D)}$ agrees with the norm $\|u\|_{L^2(\partial\Omega)}$. This is the purpose of the next result.

\begin{lemma} \label{lemma:f_is_Hardy1}
	Let $\Omega\subset\R^2$ be a simply connected bounded domain with $C^2$ boundary, and let $F:\overline{\mathbb D} \to \overline\Omega$ be a $C^1(\overline{\mathbb D})$ conformal map with $F({\mathbb D})=\Omega$ and $F(\partial{\mathbb D})=\partial\Omega$. Then, for every $u\in E(\Omega)$ such that $\partial_{\bar z}u=0$ in $\Omega$, the function $f:= u(F)(\partial_z F)^{1/2}$ belongs to $\mathcal H^2(\mathbb D)$ and 
	\begin{equation}\label{eq:ident_norms_omega_disk}
		\sqrt{2\pi}\|f\|_{\mathcal H^2(\mathbb D)} 
		=\|u\|_{L^2(\partial\Omega)}.
	\end{equation}
\end{lemma}

\begin{proof}
	First of all, let us mention that since $\Omega \subset \R^2\equiv\C$ is a simply connected bounded domain with $C^2$ boundary, by the Riemann mapping theorem \cite[Theorem 1 in Section 6.1.1]{Ahlfors1979} there always exists such a conformal map $F\in C^1(\overline{\mathbb D})$. 
	
	In order to prove that $f:= u(F)(\partial_z F)^{1/2}\in\mathcal H^2(\mathbb D)$, from \cite[Corollary in page 169; Section 10.1]{Duren1970}, it suffices to show that $u$ belongs to the  (Smirnov) class of functions $E^2(\Omega)$ defined in \cite[page 168; Section 10.1]{Duren1970}. More precisely, it is enough to prove that $u$ is holomorphic in $\Omega$ (which holds by assumption) and that there exists a sequence of rectifiable Jordan curves $\Gamma_1,\Gamma_2,\ldots$ in~$\Omega$, tending to the boundary in the sense that $\Gamma_n$ eventually surrounds each compact subdomain of $\Omega$, such that
	\begin{equation}\label{u_in_E2_Smirnov}
		\sup_{n}\int_{\Gamma_n}|u|^2<+\infty.
	\end{equation}
	
	In order to prove this, let us first build the sequence $\Gamma_n$.
	Let $\gamma=(\gamma_1,\gamma_2):\R/|\partial\Omega|\Z\to\partial\Omega\subset\R^2$ be an arc-length parametrization of $\partial\Omega$ with positive orientation. For every $s\in\R/|\partial\Omega|\Z$, it holds that 
	$\tau(\gamma(s))=\gamma'(s)$, $\nu(\gamma(s))=(\gamma_2'(s),-\gamma_1'(s))$, and 
	$\frac{d}{ds}\big(\nu(\gamma(s))\big)=\kappa(\gamma(s))\gamma'(s)$, where $\kappa(z)$ is the so-called signed curvature of $\partial\Omega$ at the point $z\in\partial\Omega$.
	Given an integer $n\geq1$, let $\gamma_n:\R/|\partial\Omega|\Z\to\partial\Omega\subset\R^2$ be defined by
	\begin{equation}\label{def:curve_n}
		\gamma_n(s):=\gamma(s)-{\textstyle \frac{1}{n}}\nu(\gamma(s))\quad\text{for } s\in\R/|\partial\Omega|\Z.
	\end{equation}
	Since $\Omega$ has a $C^2$ boundary, the curves 
	$\Gamma_n:=\gamma_n(\R/|\partial\Omega|\Z)$ are in $\Omega$ for all $n$ big enough, and they tend to the boundary as $n\uparrow+\infty$ in the sense described above. Moreover, 
	\begin{equation}\label{def:curve_n_curvature}
		\gamma_n'(s)=\big(1-{\textstyle \frac{1}{n}}\kappa(\gamma(s))\big)\gamma'(s)
		\quad\text{for } s\in\R/|\partial\Omega|\Z.
	\end{equation}
	
	Next we will show that \eqref{u_in_E2_Smirnov} holds for the sequence of curves $\{\Gamma_n\}_{n\geq n_0}$ if $n_0$ is large enough.
	Since $u\in E(\Omega)$ is such that $\partial_{\bar z}u=0$ in $\Omega$, by \cite[Remark~2.6]{Duran2025} (see also \cite[Theorem~21]{Antunes2021}) it holds that 
	\begin{equation}\label{eq:Cauchy_formula}
		u(z) = \frac{1}{2\pi i} \int_{\partial \Omega} \frac{u(\zeta)}{\zeta-z} \, d\zeta
		= \frac{1}{2\pi} \int_{\partial \Omega} \frac{u(\zeta)\nu(\zeta)}{\zeta-z} \, |d\zeta| = \mathcal T(u\nu)(z)
	\end{equation}
	for all $z\in \Omega$, where $|d\zeta|$ denotes integration with respect to arc-length ---thus, the second equality in \eqref{eq:Cauchy_formula} follows from the fact that $d\zeta = i\nu(\zeta) |d\zeta|$--- and, for $f:\partial\Omega\to\C$, we denoted 
	\begin{equation}
		\mathcal Tf(z):=\int_{\partial \Omega} k(z-\zeta)f(\zeta) \, |d\zeta|\quad\text{with}\quad k(z):=-\frac{1}{2\pi z}.
	\end{equation}  
	Note that the operator $\mathcal T$ is as in \cite[(3.2.2)]{Hofmann2010} and $k$ satisfies \cite[(3.2.1)]{Hofmann2010}. Therefore, \cite[Proposition 3.20]{Hofmann2010} yields
	\begin{equation} \label{eq:Bound_Nontan_Hofmann}
		\| \mathcal N (\mathcal T (u\nu)) \|_{L^2(\partial\Omega)} \leq C \|u\nu\|_{L^2(\partial\Omega)}=C \|u\|_{L^2(\partial\Omega)}
	\end{equation}
	for some $C>0$ depending only on $\Omega$. Here, and following \cite[(2.1.6)]{Hofmann2010}, we denoted by $\mathcal N$ the nontangential maximal operator defined, 
	for $w:\Omega\to\C$, by
	\begin{equation}
		\mathcal Nw(\zeta):=\sup\{|w(z)|:\,z\in\Omega,\,|\zeta-z|<2\operatorname{dist}(z,\partial\Omega)\}.
	\end{equation}
	With the estimate \eqref{eq:Bound_Nontan_Hofmann} in hand, \eqref{u_in_E2_Smirnov} follows easily, as we will see next. From \eqref{def:curve_n} and \eqref{def:curve_n_curvature}, we see that
	\begin{equation}
		\begin{split}
			\int_{\Gamma_n}|u|^2
			&=\int_0^{|\partial\Omega|}|u(\gamma_n(s))|^2|\gamma_n'(s)|\,ds
			=\int_0^{|\partial\Omega|}\big|u\big(\gamma(s)-{\textstyle \frac{1}{n}}\nu(\gamma(s))\big)\big|^2\big|\big(1-{\textstyle \frac{1}{n}}\kappa(\gamma(s))\big)\big||\gamma'(s)|\,ds.
		\end{split}
	\end{equation}
	Since $\Omega$ has $C^2$ boundary, for all $n$ large enough it holds that
	$|{\textstyle \frac{1}{n}}\kappa(\gamma(s))|\leq1$ for all $s$. Moreover, setting
	$z(s):=\gamma(s)-{\textstyle \frac{1}{n}}\nu(\gamma(s))$, it is clear that, for all $n$ large enough (uniformly in $s$ thanks to the $C^2$ regularity), $z(s)\in\Omega$ and $|\gamma(s)-z(s)|=
	\operatorname{dist}(z(s),\partial\Omega)$, and thus
	\begin{equation}
		|\gamma(s)-z(s)|
		<2\operatorname{dist}(z(s),\partial\Omega)
		\quad\text{for all $s$.}
	\end{equation}
	This last observation together with \eqref{eq:Cauchy_formula} yields
	\begin{equation}
		\big|u\big(\gamma(s)-{\textstyle \frac{1}{n}}\nu(\gamma(s))\big)\big|=|u(z(s))|=|\mathcal T(u\nu)(z(s))|\leq \mathcal N (\mathcal T(u\nu))(\gamma(s))\quad\text{for all $s$,}
	\end{equation}
	whenever $n$ is large enough. Then, combining all these estimates with \eqref{eq:Bound_Nontan_Hofmann}, we conclude that
	\begin{equation}
		\begin{split}
			\int_{\Gamma_n}|u|^2
			\leq2\int_0^{|\partial\Omega|}\big(\mathcal N (\mathcal T(u\nu))(\gamma(s))\big)^2|\gamma'(s)|\,ds
			=2\| \mathcal N (\mathcal T (u\nu)) \|_{L^2(\partial\Omega)}^2 \leq C \|u\|_{L^2(\partial\Omega)}^2
		\end{split}
	\end{equation}   
	for all $n$ large enough, where $C>0$ depends only on $\Omega$. This proves \eqref{u_in_E2_Smirnov} which, as we mentioned, yields that $f:= u(F)(\partial_z F)^{1/2}\in\mathcal H^2(\mathbb D)$. 
	
	It only remains to prove \eqref{eq:ident_norms_omega_disk}, that is, $\sqrt{2\pi}\|f\|_{\mathcal H^2(\mathbb D)} 
	=\|u\|_{L^2(\partial\Omega)}$. 
	From \eqref{Hardy_p:norm_1} and \eqref{Hardy_p:norm_2} we know that 
	\begin{equation}
		\|f\|_{\mathcal H^2(\mathbb D)}  
		=\Big(\frac{1}{2\pi}\int_0^{2\pi} \lim_{r\to1^-}|f(r e^{i\phi})|^2\, d\phi\Big)^{1/2}.
	\end{equation}
	Recall that $f(r e^{i\phi})=u(F(r e^{i\phi}))(\partial_z F(r e^{i\phi}))^{1/2}$. On the one hand, since $F\in C^1(\overline{\mathbb D})$, we have that 
	$\lim_{r\to1^-}\partial_z F(r e^{i\phi})=\partial_zF(e^{i\phi})$ for all $\phi$. On the other hand, from the comments in \cite[page 45 in Section 3.5]{Duren1970}, $F$ preserves angles at almost every boundary point, which means that $F(r e^{i\phi})$ approaches $F(e^{i\phi})$ nontangentially as $r\to1^-$ for a.e. $\phi$. Therefore, using \eqref{eq:Cauchy_formula} and  \cite[Theorem 3.32]{Hofmann2010} we deduce that the limit
	\begin{equation}\label{nontan_limit_exists}
		u_\mathrm{nt}(F(e^{i\phi})):=\lim_{r\to1^-}u(F(r e^{i\phi}))=\lim_{r\to1^-}\mathcal T(u\nu)(F(r e^{i\phi}))\quad
		\text{exists for a.e. $\phi$.}
	\end{equation} 
	Moreover, the fact that $u\in L^2(\partial\Omega)$ combined with \eqref{eq:Cauchy_formula} and \cite[Lemma~2.7]{Duran2025} shows that $u\in H^{1/2}(\Omega)$. Then, since 
	$\Delta u =4\partial_z\partial_{\bar z}u=0$ in $\Omega$, \eqref{nontan_limit_exists} and \cite[Theorem 3.6 $(ii)$]{Behrndt2025} give that $u_\mathrm{nt}(F(e^{i\phi}))=u(F(e^{i\phi}))$ for a.e. $\phi$, where, by abuse of notation, we are denoting by $u(F(e^{i\phi}))$ the Dirichlet trace of $u$ ---in the sense described below \eqref{def:E}--- evaluated at the  point $F(e^{i\phi})\in\partial\Omega$. All in all, we have shown that 
	\begin{equation}
		\lim_{r\to1^-}f(r e^{i\phi})=u(F(e^{i\phi}))(\partial_z F(e^{i\phi}))^{1/2}\quad\text{for a.e. $\phi$}.
	\end{equation} 
	Therefore, using that $\phi\mapsto \zeta:=F(e^{i\phi})\in \partial\Omega$ is a parametrization of $\partial\Omega$ and that $\partial_{\bar z}F=0$ in $\overline{\mathbb D}$ ---hence the change of variables $\phi\mapsto\zeta$ leads to
	$|\partial_z F(e^{i\phi})|\,d\phi\mapsto|d\zeta|$---, we conclude that
	\begin{equation}
		\begin{split}
			2\pi\|f\|_{\mathcal H^2(\mathbb D)}^2
			&=\int_0^{2\pi} \lim_{r\to1^-}|f(r e^{i\phi})|^2\, d\phi
			=\int_0^{2\pi} |u(F(e^{i\phi}))|^2|\partial_z F(e^{i\phi})|\,d\phi
			=\|u\|_{L^2(\partial\Omega)}^2,
		\end{split}
	\end{equation}
	which is \eqref{eq:ident_norms_omega_disk}.
\end{proof}

As a consequence of \Cref{lemma:Carleman} for $p=2$ and \Cref{lemma:f_is_Hardy1}, we can finally establish the sharp inequality of \Cref{prop:CarlemanL2}.

\begin{proof}[Proof of \Cref{prop:CarlemanL2}]
	Given $u\in E(\Omega)$ with $\partial_{\bar z}u=0$ in $\Omega$, let $F$ and $f:= u(F)(\partial_z F)^{1/2}\in\mathcal H^2(\mathbb D)$ be as in \Cref{lemma:f_is_Hardy1}. By \Cref{lemma:Carleman} for $p=2$, 
	$f\in \mathcal A^{4}(\mathbb D)$ and $	\|f\|_{\mathcal A^{4}(\mathbb D)}\leq\|f\|_{\mathcal H^2(\mathbb D)}$.
	Thus, by \Cref{lemma:f_is_Hardy1}, we have
	\begin{equation}\label{ineq:Carleman_p=2}
		\|f\|_{\mathcal A^{4}(\mathbb D)}
		\leq\|f\|_{\mathcal H^2(\mathbb D)}
		= \dfrac{1}{\sqrt{2\pi}}\|u\|_{L^2(\partial \Omega)}.
	\end{equation}
	Now, since $\partial_{\bar z}F=0$ in $\mathbb D$, by the change of variables formula in complex notation we have that
	\begin{equation}
		\int_\Omega |u|^2 = \int_{F(\mathbb D)} |u|^2 = \int_{\mathbb D} |u(F)|^2|\partial_zF|^2 = \int_{\mathbb D} |u(F)(\partial_zF)^{1/2}|^2 |\partial_zF| = \int_{\mathbb D} |f|^2 |\partial_z F|.
	\end{equation}
	Therefore, using the Cauchy-Schwarz inequality, the definition of $\|\cdot\|_{\mathcal A^{4}(\mathbb D)}$, \eqref{ineq:Carleman_p=2}, and the fact that 
	$\|\partial_zF\|_{L^2(\mathbb D)}^2=|\Omega|$, we get
	\begin{equation}
		\begin{split}
			\|u\|_{L^2(\Omega)}^2
			& \leq\|f\|_{L^4(\mathbb D)}^2\|\partial_zF\|_{L^2(\mathbb D)}
			=\sqrt{\pi}\|f\|_{\mathcal A^{4}(\mathbb D)}^2\|\partial_zF\|_{L^2(\mathbb D)}\\
			&\leq \dfrac{1}{2\sqrt{\pi}}\|u\|_{L^2(\partial \Omega)}
			\|\partial_zF\|_{L^2(\mathbb D)}
			= \frac{1}{2}\sqrt{\frac{|\Omega|}{\pi}}  
			\|u\|_{L^2(\partial\Omega)}^2,
		\end{split}
	\end{equation}
	which proves \eqref{eq:Carleman_slope}.   
	
	Moreover, from this chain of inequalities we see that the equality in \eqref{eq:Carleman_slope} holds if and only if there is equality both in the Cauchy-Schwarz inequality and \eqref{ineq:Carleman_p=2}. If $u$ does not vanish identically, the first one yields 
	$|\partial_zF|=c|f^2|$ in $\mathbb D$ for some $c>0$ and, in view of \Cref{lemma:Carleman}, the second one leads to 
	\begin{equation}\label{optimal_f_carleman}
		f(z) = \frac{c_1}{1-c_2 z}
		\quad\text{for all $z\in{\mathbb D}$,}
	\end{equation} 
	for some $c_1,c_2 \in \C$ with $c_1\neq0$ and $|c_2| < 1$. Note that the equality $|\partial_zF|=c|f^2|$ means that $\partial_zF=c e^{ih}f^2$ in $\mathbb D$ for some $h:\mathbb D\to\R$. Using \eqref{optimal_f_carleman}, this in turn implies that 
	\begin{equation}
		e^{ih(z)}=\frac{\partial_zF(z)}{cf(z)^{2}}=\frac{(1-c_2 z)^2}{cc_1^2}\partial_zF(z) \quad\text{for all $z\in{\mathbb D}$.}
	\end{equation}
	Since the right-hand side is a holomorphic function in $\mathbb D$ but the left-hand side takes values in $\partial \mathbb D$ (which has empty interior), the open mapping theorem for holomorphic functions forces $h$ to be a constant function. Therefore, we deduce that
	\begin{equation}\label{eq_CS-Carleman}
		\partial_zF(z) =  \frac{c_3}{(1-c_2 z)^2} \quad
		\text{for all $z\in{\mathbb D}$ and  some $c_2,c_3 \in \C$ with $c_3\neq0$ and $|c_2| < 1$.}
	\end{equation}  
	
	To conclude, we distinguish two cases.
	Assume first that $c_2=0$ in \eqref{eq_CS-Carleman}. Then, using that $\partial_{\bar z}F=0$, we deduce that $F(z)=c_3z+c_4$ for some $c_3,c_4 \in \C$ with $c_3\neq0$. This shows that $\Omega=F(\mathbb D)$ is a disk. Moreover, 
	since $h$ is constant, we have
	$\partial_zF = c'f^2 = c' u(F)^2\partial_zF$
	with $c':=c e^{ih}\in\C$,
	which forces $u(F)^2$ and $u$ to be constant functions in 
	${\mathbb D}$ and $\Omega$, respectively.

	Assume now that $c_2\neq0$. Then \eqref{eq_CS-Carleman} gives
	\begin{equation}
		\partial_zF(z) =  \frac{c_3}{(1-c_2 z)^2} 
		=\frac{c_3}{c_2}\,\partial_z\Big(\frac{1}{1-c_2 z}\Big),
	\end{equation}
	which, using that $\partial_{\bar z}F=0$, means that
	\begin{equation}
		F(z) = \frac{c_4}{1-c_2 z} + c_5
		\quad\text{for all $z\in{\mathbb D}$,}
	\end{equation}
	for some $c_2, c_4, c_5 \in \C$ 
	with $c_4\neq0$ and $|c_2| < 1$. It is well known that such transformations $F$ carry disks onto disks or half-planes (see, for example, \cite[Theorem 14 in Section 3.3]{Ahlfors1979}). Since $\Omega$ is bounded, we conclude that 
	$\Omega=F(\mathbb D)$ is a disk. Now, as before, the equalities $\partial_zF = c'f^2 = c' u(F)^2\partial_zF$
	force $u$ to be a constant function in 
	$\Omega$.    
	
	In conclusion, we have seen that if $u$ does not vanish identically and the equality in \eqref{eq:Carleman_slope} holds then $\Omega$ is a disk and $u$ is a constant function. Conversely, if $\Omega$ is a disk of radius $r>0$ and $u$ is constant in $\Omega$,
	we get
	\begin{equation}
		\begin{split}
			\|u\|_{L^2(\Omega)}^2
			=|u|^2|\Omega| 
			=|u|^2\pi r^2
			=\frac{1}{2}\sqrt{\frac{\pi r^2}{\pi}} 
			|u|^2 2\pi r 
			= \frac{1}{2}\sqrt{\frac{|\Omega|}{\pi}} 
			|u|^2|\partial \Omega| 
			= \frac{1}{2}\sqrt{\frac{|\Omega|}{\pi}} 
			\|u\|_{L^2(\partial\Omega)}^2,
		\end{split}
	\end{equation}    
	as desired.
\end{proof}

\section{Quantum dots with negative mass} \label{sec:QD_neg_mass}

This section is mainly devoted to the proof of \Cref{thm:shape_opt_neg_mass}. 
Namely, assuming that $m<0$, we will give a characterization of the smallest $\theta\in(-\frac \pi 2, \frac \pi 2)$ for which $|m|$ is an eigenvalue of $\D_\theta$. 
More precisely, we will show that
\begin{equation}\label{eq:opt_theta_neg_mass_S_OmegaPROOF}
	\begin{split}
		\min\big\{\theta\in{\textstyle(-\frac \pi 2, \frac \pi 2)}:\,\eqref{eq:neg_mass_dirac_ev|m|}\text{ has a nonzero solution}\big\}
		=\vartheta^{-1}\bigg(\frac{2|m|}{S_\Omega}\bigg);
	\end{split}
\end{equation}
this is \eqref{eq:opt_theta_neg_mass_S_Omega} in \Cref{thm:shape_opt_neg_mass}.
As we see, this characterization turns out to be related with the constant $S_\Omega$, defined in \eqref{def:S_Omega}. 
Then, the optimality of the disk among simply connected domains will follow from \Cref{prop:muPrimeOrigin}. 
Finally, by unitary equivalence we can rewrite the result for $\theta\in(\frac \pi 2, \frac{3\pi}{2})$ and $m> 0$ as stated in \Cref{thm:shape_opt_neg_mass_unit_equiv}.

\begin{proof}[Proof of \Cref{thm:shape_opt_neg_mass}]
Let us address the proof of \eqref{eq:opt_theta_neg_mass_S_Omega}. As a first step, we will show that
\begin{equation}\label{ineq:opt_theta_neg_mass}
\begin{split}
\inf\big\{\theta\in{\textstyle(-\frac \pi 2, \frac \pi 2)}:\,\eqref{eq:neg_mass_dirac_ev|m|}\text{ has a nonzero solution}\big\}
\geq\vartheta^{-1}\bigg(\frac{2|m|}{S_\Omega}\bigg).
\end{split}
\end{equation}
Assume that, for a given $\theta\in(-\frac \pi 2, \frac \pi 2)$, we have a nonzero solution to the eigenvalue problem~\eqref{eq:neg_mass_dirac_ev|m|}.  
Writing the equation
$\D_\theta\varphi=|m|\varphi$ and the boundary condition of $\varphi\in\Dom(\D_\theta)$ in terms of its components $\varphi= (u,v)^\intercal$ we get
\begin{equation}\label{eq:sist_neg_mass_dirac_ev|m|}
    \begin{cases}
        -i \partial_z v = |m| u & \text{in } L^2(\Omega), \\
        \partial_{\bar z} u = 0 & \text{in } L^2(\Omega), \\
        v = i \vartheta(\theta) \nu u & \text{in } H^{1/2}(\partial \Omega).
    \end{cases}
\end{equation}
If we multiply the first equation by $\overline u$, integrate by parts, use that 
$\partial_{\bar z} u = 0$ in $\Omega$ by the second equation, and apply the boundary condition from the third equation, we end up with
\begin{equation}\label{eq:RayQuo_neg_mass_ev|m|}
\begin{split}
|m|\int_\Omega|u|^2
&=-i\int_\Omega \partial_{ z} v\,\overline u
=i\int_\Omega  v\,\overline{\partial_{\bar z}u}
-\frac{i}{2}\int_{\partial\Omega}\overline \nu v \overline u
=\frac{\vartheta(\theta)}{2}\int_{\partial\Omega}|u|^2.
\end{split}
\end{equation}
Note that, since $m<0$ and $\varphi$ is assumed to be a nonzero solution, we must have $\int_{\partial\Omega}|u|^2>0$. Otherwise, \eqref{eq:RayQuo_neg_mass_ev|m|} would imply that $u=0$ in $\Omega$, and then the first and third equation in \eqref{eq:sist_neg_mass_dirac_ev|m|} would imply that $\partial_z v=0$ in $\Omega$ and $v=0$ on $\partial\Omega$ respectively, which would lead to $v=0$ in~$\Omega$ and, thus, $\varphi=0$ in~$\Omega$, reaching a contradiction. Therefore, from \eqref{eq:RayQuo_neg_mass_ev|m|}, the second equation in~\eqref{eq:sist_neg_mass_dirac_ev|m|}, and the definition of $S_\Omega$ \eqref{def:S_Omega}, we deduce that
\begin{equation}
\vartheta(\theta)
=2|m|\frac{\int_{\Omega}|u|^2}{\int_{\partial\Omega}|u|^2}
\leq2|m|\sup_{w\in E(\Omega)\setminus\{0\}:\,
\partial_{\bar z}w=0\text{ in }\Omega}
\frac{\int_{\Omega}|w|^2}{\int_{\partial\Omega}|w|^2}
=\frac{2|m|}{S_\Omega}
\leq|m| \sqrt{\frac{|\Omega|}{\pi}}.
\end{equation}
From this, and using also that $\vartheta^{-1}$ is strictly decreasing in $(0,+\infty)$, we arrive to \eqref{ineq:opt_theta_neg_mass}.

The next step will be to prove that the infimum on the left-hand side of \eqref{ineq:opt_theta_neg_mass} is attained and that the inequality is actually an equality, which is precisely what \eqref{eq:opt_theta_neg_mass_S_Omega} states. 
For this purpose, set 
\begin{equation} \label{eq:theta_omega}
    \theta_\Omega:=\vartheta^{-1}\big(\frac{2|m|}{S_\Omega}\big),
\end{equation}
and recall from \Cref{prop:muPrimeOrigin}~$(i)$ that 
\begin{equation}\label{eq:S_Omega_u_optimal}
S_\Omega
=\frac{\int_{\partial\Omega}|u_\Omega|^2}{\int_{\Omega}|u_\Omega|^2}\quad
\text{for some $u_\Omega\in H^1(\Omega)\setminus\{0\}$ with
$\partial_{\bar z}u_\Omega=0$ in $\Omega$},
\end{equation}
and that 
\begin{equation} \label{eq:weaku_omega}
    \int_{\partial\Omega} u_\Omega \, \overline w = S_\Omega \int_\Omega u_\Omega \, \overline w \quad \text{for any } w\in E(\Omega)\setminus\{0\} \text{ with } \partial_{\bar z}w=0\text{ in }\Omega.
\end{equation}
We will see that $u_\Omega$ yields a nonzero solution to \eqref{eq:sist_neg_mass_dirac_ev|m|} for 
$\theta_\Omega$ and a suitably chosen function $v_\Omega$. Since \eqref{eq:sist_neg_mass_dirac_ev|m|} and \eqref{eq:neg_mass_dirac_ev|m|} are equivalent formulations of the same problem, this means that we will have a nonzero solution to \eqref{eq:neg_mass_dirac_ev|m|} for $\theta=\theta_\Omega$. Once this is shown,  \eqref{eq:opt_theta_neg_mass_S_Omega} follows from \eqref{ineq:opt_theta_neg_mass}.

For $\theta_\Omega$ and $u_\Omega$ as in \eqref{eq:theta_omega} and \eqref{eq:S_Omega_u_optimal}, let $v_\Omega$ be the unique solution in $H^1(\Omega)$ to the Dirichlet problem
\begin{equation}\label{eq:Dirich_v_Omega}
\begin{cases}
        \Delta v_\Omega = 0 & \text{in } \Omega, \\
        v_\Omega = i \vartheta(\theta_\Omega) \nu u_\Omega & \text{in } H^{1/2}(\partial \Omega).
\end{cases}
\end{equation}
Since $\partial_{\bar z}u_\Omega=0$ in $\Omega$ and $v_\Omega = i \vartheta(\theta_\Omega) \nu u_\Omega$ in $H^{1/2}(\partial \Omega)$, to get \eqref{eq:sist_neg_mass_dirac_ev|m|} we only need to show  that 
\begin{equation}\label{eq:last_one_for_sist_neg_mass_dirac_ev|m|}
-i \partial_z v_\Omega = |m| u_\Omega\quad\text{in }L^2(\Omega).
\end{equation} 
In order to prove this, recall that the Bergman space 
\begin{equation}
\mathcal A^2(\Omega):=\{f\in L^2(\Omega):\,\partial_{\bar z}f=0\text{ in }\Omega\}
\end{equation} 
is a closed subspace of $L^2(\Omega)$ (see \cite[Proposition~1.13]{Conway2007}) and, thus, the orthogonal projections $P:L^2(\Omega)\to \mathcal A^2(\Omega)\subset L^2(\Omega)$ and $P_{\bot}:=\id-P:L^2(\Omega)\to L^2(\Omega)$ are well-defined  bounded self-adjoint operators in $L^2(\Omega)$. In particular, given $h\in C^\infty_c(\Omega) \subset E(\Omega)$, by \cite[Lemma~2.10]{Duran2025} we have that $h=Ph+P_{\bot}h$ and 
\begin{equation}\label{eqs:orth_proj}
\begin{split}
\partial_{\bar z}(Ph)=0 \ \text{ in $\Omega$ with } Ph\in L^2(\partial\Omega),
\qquad\int_\Omega w\overline{P_{\bot}h}=0\quad\text{for all 
$w\in \mathcal A^2(\Omega)$}.
\end{split}
\end{equation} 
With these ingredients in hand, the proof of \eqref{eq:last_one_for_sist_neg_mass_dirac_ev|m|} follows easily. Indeed, by integration by parts, we see that
\begin{equation}\label{eq:int_parts_last_one}
\begin{split}
\int_\Omega(-i\partial_z v_\Omega)\overline h
&=-i\int_\Omega\partial_z v_\Omega\overline{Ph}
-i\int_\Omega\partial_z v_\Omega\overline{P_\bot h}\\
&=i\int_\Omega v_\Omega\overline{\partial_{\bar z}(Ph)}
-\frac{i}{2}\int_{\partial\Omega}
\overline\nu v_\Omega\overline{P h}
-i\int_\Omega\partial_z v_\Omega\overline{P_\bot h}.
\end{split}
\end{equation}
The first term in the right-hand side of \eqref{eq:int_parts_last_one} vanishes by the first equality in \eqref{eqs:orth_proj}, while the last term in the right-hand side of \eqref{eq:int_parts_last_one} vanishes by the last equality in \eqref{eqs:orth_proj} and the fact that 
$\partial_z v_\Omega\in \mathcal A^2(\Omega)$, since 
$v_\Omega\in H^1(\Omega)$ and
$\partial_{\bar z}(\partial_z v_\Omega)=\frac 1 4\Delta v_\Omega=0$ in $\Omega$ by \eqref{eq:Dirich_v_Omega}. Therefore, applying this, the boundary condition of \eqref{eq:Dirich_v_Omega}, \eqref{eq:weaku_omega}, the definition of $\theta_\Omega$ in \eqref{eq:theta_omega}, and the last equality in \eqref{eqs:orth_proj} to \eqref{eq:int_parts_last_one} leads to 
\begin{equation}
\begin{split}
\int_\Omega(-i\partial_z v_\Omega)\overline h
&=\frac{\vartheta(\theta_\Omega)}{2}\int_{\partial\Omega}
u_\Omega\overline{P h}
=\frac{\vartheta(\theta_\Omega)}{2}S_\Omega
\int_{\Omega}u_\Omega\overline{P h}\\
&=|m|\int_{\Omega}u_\Omega\overline{P h}
=|m|\int_{\Omega}u_\Omega\overline{(P h+P_\bot h)}
=\int_{\Omega}(|m|u_\Omega)\overline{h}.
\end{split}
\end{equation}
Since this holds for all $h\in C^\infty_c(\Omega)$, \eqref{eq:last_one_for_sist_neg_mass_dirac_ev|m|} follows by a density argument, and the proof of \eqref{eq:opt_theta_neg_mass_S_Omega} is complete.

The proof of \eqref{eq:shape_opt_theta_neg_mass_S_Omega} follows easily from \eqref{eq:opt_theta_neg_mass_S_Omega}, \eqref{eq:S_Omega_u_optimal}, \Cref{prop:CarlemanL2}, and the fact that $\vartheta^{-1}$ is a strictly decreasing function. 
\end{proof}

As mentioned in the introduction, using the unitary equivalence between the operators $\D_\theta(m)$ and $-\D_{\pi-\theta}(-m)$ (see \Cref{sec:Invariance}), \Cref{thm:shape_opt_neg_mass} leads to a shape optimization result for $\D_\theta(m)$ when $m>0$ and $\theta\in(\frac \pi 2, \frac {3\pi} {2})$.

\begin{proof}[Proof of \Cref{thm:shape_opt_neg_mass_unit_equiv}]
Assume that $\varphi$ is a nonzero solution to \eqref{eq:ev_prob_QD_other_branch}, that is, to
\begin{equation}
	\begin{cases}
		\varphi\in\Dom\big(\D_\theta(m)\big), & \\
		\D_\theta(m)\varphi=-m\varphi & \text{in } L^2(\Omega)^2.
	\end{cases}
\end{equation}
Then, setting 
$\psi:=\sigma_3\varphi$ ---recall that  
$\psi\in\Dom\big(\D_{\pi-\theta}(-m)\big)$ by \eqref{eq:unit_equiv_bdy_cond}---,
from \eqref{eq:unit_equiv} we see that
\begin{equation}
\D_{\pi-\theta}(-m)\psi=
-\sigma_3\D_\theta(m)\sigma_3\psi=
-\sigma_3\D_\theta(m)\varphi=
m\sigma_3\varphi=m\psi=|m|\psi.
\end{equation}
Thus, we get a nonzero solution to the eigenvalue problem
\begin{equation}\label{eq:ev_prob_QD_other_branch_2}
\begin{cases}
\psi\in\Dom\big(\D_{\theta^*}(-m)\big), & \\
\D_{\theta^*}(-m)\psi=|m|\psi & \text{in } L^2(\Omega)^2
\end{cases}
\end{equation}
with $\theta^*:=\pi-\theta\in{\textstyle(-\frac \pi 2, \frac \pi 2)}$. That a nonzero solution to \eqref{eq:ev_prob_QD_other_branch_2} leads to a nonzero solution to \eqref{eq:ev_prob_QD_other_branch} follows analogously. With these observations in hand, the proof of the theorem follows directly from \Cref{thm:shape_opt_neg_mass} applied to the eigenvalue problem \eqref{eq:ev_prob_QD_other_branch_2}, taking into account the relation $\theta=\pi-\theta^*$.
\end{proof}

\appendix

\section{} 

\subsection{Notation}\label{sec:Notation}
In this section we recall some basic notation used within the paper. Throughout the work, $\Omega$ denotes a bounded domain in $\R^2$ with $C^2$~boundary. Regarding integration, we consider the natural measures depending on the situation: the Lebesgue measure on $\Omega$ and the surface (arc-length) measure on $\partial\Omega$. However, the reader should be aware that in \Cref{sec:asympt_regimes} we also make use of the line integral on $\Omega$ in the complex sense. For the sake of simplicity, when computing integrals, we omit the measure of integration if it is clear from the context.

We denote by 
$L^2(\Omega)$ the Hilbert space of  functions 
$u:\Omega\to\C$ endowed with the scalar product and the associated norm
\begin{equation}
\langle u,v\rangle_{L^2(\Omega)}:=\int_{\Omega} u\,\overline v \quad\text{and}\quad
\|u\|_{L^2(\Omega)}:=\sqrt{\langle u,u\rangle_{L^2(\Omega)}},
\end{equation}
respectively. We denote by
$H^1(\Omega)$ the Sobolev space of 
functions in $L^2(\Omega)$ with first weak partial derivatives in $L^2(\Omega)$. 

Similarly, $L^2(\partial\Omega)$ denotes the Hilbert space of  functions $u:\partial\Omega\to\C$ endowed with the scalar product and the associated norm
\begin{equation}
\langle u,v\rangle_{L^2(\partial\Omega)}:=\int_{\partial\Omega} u\,\overline v \quad\text{and}\quad
\|u\|_{L^2(\partial\Omega)}:=\sqrt{\langle u,u\rangle_{L^2(\partial\Omega)}},
\end{equation}
respectively.
We denote by $H^{1/2}(\partial \Omega)$ the fractional Sobolev space of functions $u\in L^2(\partial\Omega)$ such that 
    $$
        \|u\|_{H^{1/2}(\partial \Omega)}:= \Big( \int_{\partial\Omega} |u|^2 + \int_{\partial\Omega}\int_{\partial\Omega}\frac{|u(x)-u(y)|^2}{|x-y|^{2}} |dy| |dx| \Big)^{1/2} 
       < +\infty.
    $$
Here, $|dx|$ and $|dy|$ denote integration with respect to arc-length\footnote{We use the same notation as in \cite{Duren1970} in order to distinguish the integration with respect to arc-length (denoted by $|dz|$ for $z\in\partial \Omega$) and the line integral in complex variables (denoted by $dz$ for $z\in\partial \Omega$ and used in \Cref{sec:asympt_regimes}).} on $\partial\Omega$.

    The continuous dual of $H^{1/2}(\partial \Omega)$ is denoted by $H^{-1/2}(\partial \Omega)$. The action of $u\in H^{-1/2}(\partial \Omega)$ on $v \in H^{1/2}(\partial \Omega)$ is denoted by $\langle u, v \rangle_{H^{-1/2}(\partial\Omega), H^{1/2}(\partial\Omega)}$, and the norm in $H^{-1/2}(\partial \Omega)$ is
    $$
        \|u\|_{H^{-1/2}(\partial \Omega)}:= 
        { \sup_{\|u\|_{H^{1/2}(\partial \Omega)}\leq1}} 
        \langle u, v \rangle_{H^{-1/2}(\partial\Omega), H^{1/2}(\partial\Omega)} .
    $$
    Recall that
\begin{equation}\label{eq:dual_pairing_L2}
\langle u, v \rangle_{H^{-1/2}(\partial\Omega), H^{1/2}(\partial\Omega)}=\langle v, u \rangle _{L^2(\partial\Omega)}
\end{equation}
    whenever $u\in L^2(\partial\Omega)\subset H^{-1/2}(\partial\Omega)$ and $v\in H^{1/2}(\partial\Omega)\subset L^2(\partial\Omega)$; see for example \cite[Remark~3 in Section 5.2]{Brezis2011}. The reason why the functions $u$ and $v$ do not appear in the same order in both sides of \eqref{eq:dual_pairing_L2} is that we defined 
$\langle \cdot, \cdot \rangle _{L^2(\partial\Omega)}$ to be linear with respect to the first entry.

Finally, we denote $X^2:=X\times X$ for $X$ any of the spaces defined above, with the natural scalar products and norms.

\subsection{Invariances of the Dirac operator}\label{sec:Invariance}
Here we give a detailed explanation of the unitary equivalences that we have used in this work.
For this, recall that $\D_\theta(m)$ is defined by
\begin{equation}
	\begin{split}
		\mathrm{Dom}\big(\D_\theta(m)\big) &:= \big\{ \varphi \in H^1(\Omega)^2: \, \varphi =(\cos\theta\,\sigma\cdot \tau+\sin\theta\, \sigma_3) \varphi  \,\text{ in } H^{1/2}(\partial \Omega)^2 \big\},\\
		\D_\theta(m)\varphi &:= 
		(-i\sigma\cdot\nabla+m\sigma_3)\varphi \quad\text{for all 
			$\varphi\in\mathrm{Dom}\big(\D_\theta(m)\big).$}
	\end{split}
\end{equation}

\medspace

$\bullet$ \textbf{Charge conjugation.}
Assume that $\varphi\in \mathrm{Dom}\big(\D_\theta(m)\big)$ and consider the charge conjugation 
\begin{equation}
	\psi:=\sigma_1\overline\varphi\in H^1(\Omega)^2.
\end{equation}
From the boundary condition that $\varphi$ satisfies (and the fact that 
$\sigma_1\overline{\sigma_2}={\sigma_2}\sigma_1$), we see that
\begin{equation}
	\begin{split}
		\psi&=\sigma_1\overline\varphi
		=\sigma_1\overline{(\cos\theta\,\sigma\cdot \tau+\sin\theta\, \sigma_3)\varphi}
		=\sigma_1(\cos\theta\,\sigma_1 \tau_1+\cos\theta\,\overline{\sigma_2} \tau_2+\sin\theta\, \sigma_3)\overline\varphi\\
		&=(\cos\theta\,\sigma_1 \tau_1+\cos\theta\,{\sigma_2} \tau_2-\sin\theta\, \sigma_3)\sigma_1\overline\varphi
		=(\cos(-\theta)\,\sigma\cdot \tau
		+\sin(-\theta)\, \sigma_3)\psi,
	\end{split}
\end{equation}
which shows that $\psi\in\Dom\big(\D_{-\theta}(m)\big)$. In addition, since $\overline{\sigma_2}=-\sigma_2$,
\begin{equation}
	\begin{split}
		\D_{-\theta}(m)\psi
		&=(-i \sigma_1\partial_1-i\sigma_2\partial_2 + m \sigma_3)\sigma_1\overline\varphi
		=\sigma_1(-i \sigma_1\partial_1+i\sigma_2\partial_2 - m \sigma_3)\overline\varphi\\
		&=\sigma_1\overline{(i \sigma_1\partial_1-i
			\overline{\sigma_2}\partial_2 - m \sigma_3)\varphi}
		=\sigma_1\overline{(i \sigma_1\partial_1+i
			{\sigma_2}\partial_2 - m \sigma_3)\varphi}\\
		&=-\sigma_1\overline{(-i \sigma\cdot\nabla + m \sigma_3)\varphi}
		=-\sigma_1\overline{\D_\theta(m)\varphi}.
	\end{split}
\end{equation}
Since these arguments are reversible, we deduce that $\varphi$ is an eigenfunction of $\D_\theta(m)$ with eigenvalue $\lambda \in \R$ if and only if $\psi$ is an eigenfunction of $\D_{-\theta}(m)$ with eigenvalue $-\lambda$.

Note that this invariance is precisely illustrated by the odd symmetry with respect to $0$ and $\pi$ in \Cref{Fig:EVCurvesQDDiracFullRange} and \Cref{Fig:EVCurvesQDDirac(NegativeMass)}.

\medspace

$\bullet$ \textbf{Chiral transformation.}
Assume that $\varphi\in \mathrm{Dom}\big(\D_\theta(m)\big)$ and consider
\begin{equation}
	\psi:=\sigma_3\varphi \in H^1(\Omega)^2.
\end{equation}
From the boundary condition that $\varphi$ satisfies (and the fact that $\sigma_3\sigma =-\sigma  \sigma_3$), we see that
\begin{equation}\label{eq:unit_equiv_bdy_cond}
	\begin{split}
		\psi&=\sigma_3\varphi=\sigma_3(\cos\theta\,\sigma\cdot \tau+\sin\theta\, \sigma_3) \varphi
		=(-\cos\theta\,\sigma\cdot \tau+\sin\theta\, \sigma_3) \sigma_3\varphi\\
		&=(-\cos\theta\,\sigma\cdot \tau+\sin\theta\, \sigma_3) \psi
		=(\cos(\pi-\theta)\,\sigma\cdot \tau+\sin(\pi-\theta)\, \sigma_3) \psi.
	\end{split}
\end{equation}
In addition, we have
\begin{equation}
	\sigma_3(-i\sigma\cdot\nabla+m\sigma_3)\sigma_3
	=-(-i\sigma\cdot\nabla-m\sigma_3). 
\end{equation}
These two simple facts show that 
\begin{equation}\label{eq:unit_equiv}
	\sigma_3\D_\theta(m)\sigma_3=-\D_{\pi-\theta}(-m)
	\quad\text{for all $\theta,m\in\R$,}
\end{equation} 
which means that $\D_\theta(m)$ is unitarily equivalent to $-\D_{\pi-\theta}(-m)$ ---this unitary equivalence was already considered in \cite[Lemma 3.2]{Holzmann2021}. Therefore, the spectral study of $\D_\theta(m)$ for  
$\theta\in(-\frac \pi 2, \frac {3\pi} {2})\setminus\{\frac \pi 2\}$ reduces to the one of 
$\D_\theta(m)$ and $\D_\theta(-m)$  for 
$\theta\in(-\frac \pi 2, \frac {\pi} {2})$, since 
$-\frac \pi 2<\pi-\theta<\frac {\pi} {2}$ whenever 
$\frac \pi 2<\theta<\frac {3\pi} {2}$.
In particular, $\varphi$ is an eigenfunction of $\D_\theta(m)$ with eigenvalue $\lambda \in \R$ if and only if $\psi$ is an eigenfunction of $\D_{\pi-\theta}(m)$ with eigenvalue $-\lambda$.

Note that this invariance is precisely illustrated by comparing \Cref{Fig:EVCurvesQDDiracFullRange} and \Cref{Fig:EVCurvesQDDirac(NegativeMass)}.

\end{document}